\documentclass{amsart}
\usepackage{amsfonts}

\newtheorem{theorem}{Theorem}
\theoremstyle{plain}

\newtheorem{definition}{Definition}
\newtheorem{example}{Example}

\newtheorem{lemma}{Lemma}

\newtheorem{proposition}{Proposition}
\newtheorem{remark}{Remark}

\numberwithin{equation}{section}
\input{tcilatex}

\begin{document}
\title[]{On the non-stationary Navier-Stokes type flows and reiterated
homogenization}
\author{Lazarus Signing}
\address{University of Ngaoundere, Department of Mathematics and Computer
Science, P.O.Box 454 Ngaoundere (Cameroon)}
\email{lsigning@yahoo.fr}
\dedicatory{}
\thanks{}

\begin{abstract}
We study the deterministic reiterated homogenization of the non-stationary
Navier-Stokes type equations in fixed domains with periodically rapidly
varying coefficients. One convergence theorem and a corrector result are
proved, and we derive the macroscopic homogenized model.
\end{abstract}

\maketitle

\section{Introduction}

We denote by $\mathbb{R}_{y}^{N}$ (resp. $\mathbb{R}_{z}^{N}$) the $N$%
-dimensional numerical space $\mathbb{R}^{N}$ of variable $y=\left(
y_{1},...,y_{N}\right) $ (resp. $z=\left( z_{1},...,z_{N}\right) $). Let $%
\left( a_{ij}\right) _{i,j=1,...,N}$ be the family of functions of $\mathbb{R%
}_{y}^{N}\times \mathbb{R}_{z}^{N}$ into $\mathbb{R}$ such that:%
\begin{equation}
a_{ij}\in \mathcal{B}\left( \mathbb{R}_{y}^{N}\times \mathbb{R}%
_{z}^{N}\right) \text{\qquad }\left( 1\leq i,j\leq N\right) \text{,}
\label{eq1.1a}
\end{equation}

\begin{equation}
a_{ij}=a_{ji}\text{\qquad }\left( 1\leq i,j\leq N\right)  \label{eq1.1b}
\end{equation}%
and there exists a constant $\alpha >0$ verifying%
\begin{equation}
\sum_{i,j=1}^{N}a_{ij}\left( y,z\right) \xi _{j}\xi _{i}\geq \alpha
\left\vert \xi \right\vert ^{2}\text{\qquad }\left( \xi =\left( \xi
_{j}\right) \in \mathbb{R}^{N}\text{, }\left( y,z\right) \in \mathbb{R}%
^{N}\times \mathbb{R}^{N}\right)  \label{eq1.1c}
\end{equation}%
($\mathcal{B}\left( \mathbb{R}_{y}^{N}\times \mathbb{R}_{z}^{N}\right) $
being the vector space of continuous and bounded functions of $\mathbb{R}%
_{y}^{N}\times \mathbb{R}_{z}^{N}$ into $\mathbb{R}$).

Let us set $Q=\Omega \times ]0,T[$, where $\Omega $ is an open bounded set
in $\mathbb{R}^{N}$ and $T$ $>0$ a real number. For any real number $%
\varepsilon >0$, let $a_{ij}^{\varepsilon }$ $\left( 1\leq i,j\leq N\right) $
be the function of $\overline{\Omega }$ into $\mathbb{R}$ defined by 
\begin{equation*}
a_{ij}^{\varepsilon }\left( x\right) =a_{ij}\left( \frac{x}{\varepsilon },%
\frac{x}{\varepsilon ^{2}}\right) \text{\qquad }\left( x\in \overline{\Omega 
}\right) \text{.}
\end{equation*}%
The family $\left( a_{ij}^{\varepsilon }\right) _{1\leq i,j\leq N}$
verifies: $a_{ij}^{\varepsilon }\in C\left( \overline{\Omega }\right) $\quad 
$\left( 1\leq i,j\leq N\right) $ and 
\begin{equation}
\sum_{i,j=1}^{N}a_{ij}^{\varepsilon }\left( x\right) \xi _{j}\xi _{i}\geq
\alpha \left\vert \xi \right\vert ^{2}\text{\qquad }\left( \xi \in \mathbb{R}%
^{N}\text{, }x\in \overline{\Omega }\right) \text{.}  \label{eq1.1d}
\end{equation}%
To the coefficients $a_{ij}^{\varepsilon }$\quad $\left( 1\leq i,j\leq
N\right) $, we associate the differential operator $P^{\varepsilon }$
defined by 
\begin{equation*}
P^{\varepsilon }=-\sum_{i,j=1}^{N}\frac{\partial }{\partial x_{i}}\left(
a_{ij}^{\varepsilon }\frac{\partial }{\partial x_{j}}\right) \text{.}
\end{equation*}%
We may notice that, the operator $P^{\varepsilon }$ acts on vector functions 
$\mathbf{u}=\left( u^{i}\right) \in H^{1}\left( \Omega \right) ^{N}$ in a 
\textit{diagonal way}, i.e., 
\begin{equation*}
\left( P^{\varepsilon }\mathbf{u}\right) ^{i}=P^{\varepsilon }u^{i}\text{ }%
\left( i=1,...,N\right) \text{.}
\end{equation*}

For any $\mathbf{f=}\left( f^{j}\right) \in L^{2}\left( 0,T;H^{-1}\left(
\Omega \right) ^{N}\right) $ and any $\varepsilon >0$, we consider the
following Cauchy-Dirichlet boundary value problem:%
\begin{equation}
\frac{\partial \mathbf{u}_{\varepsilon }}{\partial t}+P^{\varepsilon }%
\mathbf{u}_{\varepsilon }+\sum_{j=1}^{N}u_{\varepsilon }^{j}\frac{\partial 
\mathbf{u}_{\varepsilon }}{\partial x_{j}}+\mathbf{grad}p_{\varepsilon }=%
\mathbf{f}\text{ in }Q\qquad \qquad \qquad  \label{eq1.2a}
\end{equation}%
\begin{equation}
\qquad \qquad \qquad \qquad div\mathbf{u}_{\varepsilon }=0\text{ in }Q
\label{eq1.2b}
\end{equation}%
\begin{equation}
\qquad \qquad \qquad \qquad \qquad \qquad \qquad \mathbf{u}_{\varepsilon }=0%
\text{ on }\partial \Omega \times ]0,T[  \label{eq1.2c}
\end{equation}%
\begin{equation}
\qquad \qquad \qquad \qquad \mathbf{u}_{\varepsilon }\left( 0\right) =0\text{
in }\Omega \text{.}  \label{eq1.2d}
\end{equation}%
Let us recall the following spaces: 
\begin{equation*}
\mathcal{V}=\left\{ \mathbf{u}\in \mathcal{D}\left( \Omega ;\mathbb{R}%
\right) ^{N}:\text{ }div\mathbf{u}=0\right\} \text{,}
\end{equation*}%
\begin{equation*}
V=\text{the closure of }\mathcal{V}\text{ in }H_{0}^{1}\left( \Omega ;%
\mathbb{R}\right) ^{N}\text{,}
\end{equation*}%
\begin{equation*}
H=\text{the closure of }\mathcal{V}\text{ in }L^{2}\left( \Omega ;\mathbb{R}%
\right) ^{N}\text{.}
\end{equation*}%
We set 
\begin{equation*}
\mathcal{W}\left( 0,T\right) =\left\{ \mathbf{u}\in L^{2}\left( 0,T;V\right)
:\mathbf{u}^{\prime }\in L^{2}\left( 0,T;V^{\prime }\right) \right\} \text{,}
\end{equation*}%
where $V^{\prime }$ is the topological dual of $V$. We denote by $\left\Vert 
{\small \cdot }\right\Vert $ the norm in $V$ and $\left\vert {\small \cdot }%
\right\vert $ the norm in $H$. One has the following inclusions when
identifying $H$ to its topological dual:%
\begin{equation*}
V\subset H\subset V^{\prime }\text{,}
\end{equation*}%
each space being dense in the next. With the norm 
\begin{equation*}
\left\Vert \mathbf{u}\right\Vert _{\mathcal{W}\left( 0,T\right) }=\left(
\left\Vert \mathbf{u}\right\Vert _{L^{2}\left( 0,T;V\right) }^{2}+\left\Vert 
\mathbf{u}^{\prime }\right\Vert _{L^{2}\left( 0,T;V^{\prime }\right)
}^{2}\right) ^{\frac{1}{2}}\text{\qquad }\left( \mathbf{u}\in \mathcal{W}%
\left( 0,T\right) \right) \text{,}
\end{equation*}%
$\mathcal{W}\left( 0,T\right) $ is a Hilbert space with the following
properties \cite{bib14}: $\mathcal{W}\left( 0,T\right) $ is continuously
embedded in $\mathcal{C}\left( 0,T;H\right) $, and compactly embedded in $%
L^{2}\left( 0,T;H\right) $. The problem (\ref{eq1.2a})-(\ref{eq1.2d}) admits
a unique solution $\left( \mathbf{u}_{\varepsilon },p_{\varepsilon }\right) $
in $\mathcal{W}\left( 0,T\right) \times L^{2}\left( 0,T;L^{2}\left( \Omega ;%
\mathbb{R}\right) \mathfrak{/}\mathbb{R}\right) $ in dimension $N=2$, where $%
L^{2}\left( \Omega ;\mathbb{R}\right) \mathfrak{/}\mathbb{R=}\left\{ v\in
L^{2}\left( \Omega ;\mathbb{R}\right) :\int_{\Omega }v\left( x\right)
dx=0\right\} $ (We refer the reader for example to \cite{bib14}). We suppose
in the rest of the paper that $N=2$.

We study in this work the behaviour of the couple $\left( \mathbf{u}%
_{\varepsilon },p_{\varepsilon }\right) $ solution to (\ref{eq1.2a})-(\ref%
{eq1.2d}) when $\varepsilon $ tends to $0$, under the periodicity hypothesis
on the viscosity coefficients $\left( a_{ij}\right) $. Let us recall that a
function $u\in L_{loc}^{1}\left( \mathbb{R}_{y}^{N}\times \mathbb{R}%
_{z}^{N}\right) $ is said to be $Y\times Z$-periodic if for each $\left(
k,l\right) \in \mathbb{Z}^{N}\times \mathbb{Z}^{N}$ ($\mathbb{Z}$ denotes
the integers), we have $u\left( y+k,z+l\right) =u\left( y,z\right) $ almost
everywhere (a.e.) in $\left( y,z\right) \in \mathbb{R}^{N}\times \mathbb{R}%
^{N}$, where we set $Y=\left( -\frac{1}{2},\frac{1}{2}\right) ^{N}$ and $%
Z=\left( -\frac{1}{2},\frac{1}{2}\right) ^{N}$, $Y$ and $Z$ being considered
as subsets of $\mathbb{R}_{y}^{N}$ (the space $\mathbb{R}^{N}$ of variables $%
y=\left( y_{1},...,y_{N}\right) $) and $\mathbb{R}_{z}^{N}$ (the space $%
\mathbb{R}^{N}$ of variables $z=\left( z_{1},...,z_{N}\right) $),
respectively.

The study of this problem turns out to be of benefit to the modelling of
multi-phase flows with spatially varying viscosities. The homogenization of (%
\ref{eq1.2a})-(\ref{eq1.2d}) was first studied by the author \cite{bib11}
under the periodic hypothesis on the coefficients $a_{ij}$ with the scaling 
\begin{equation}
a_{ij}^{\varepsilon }\left( x\right) =a_{ij}\left( \frac{x}{\varepsilon }%
\right) \text{\qquad }\left( x\in \Omega \right) \text{,}  \label{eq1.3a}
\end{equation}%
the $a_{ij}$ belonging to $L^{\infty }\left( \mathbb{R}_{y}^{N};\mathbb{R}%
\right) $. Also, the steady state version of (\ref{eq1.2a})-(\ref{eq1.2d})
with the scaling%
\begin{equation}
a_{ij}^{\varepsilon }\left( x\right) =a_{ij}\left( \frac{x}{\varepsilon },%
\frac{x}{\varepsilon ^{2}}\right) \text{\qquad }\left( x\in \Omega \right)
\label{eq1.3b}
\end{equation}%
and $a_{ij}$ belonging to $L^{\infty }\left( \mathbb{R}_{y}^{N};\mathcal{B}%
\left( \mathbb{R}_{z}^{N};\mathbb{R}\right) \right) $ $\left( 1\leq i,j\leq
N\right) $ has been investigated by the author in \cite{bib12}. The author
presented a detailed mathematical analysis of 
\begin{equation*}
\left\{ 
\begin{array}{c}
P^{\varepsilon }\mathbf{u}_{\varepsilon }+\sum_{j=1}^{N}u_{\varepsilon }^{j}%
\frac{\partial \mathbf{u}_{\varepsilon }}{\partial x_{j}}+\mathbf{grad}%
p_{\varepsilon }=\mathbf{f}\text{ in }\Omega \text{,} \\ 
\qquad \qquad \qquad \qquad \qquad \quad div\mathbf{u}_{\varepsilon }=0\text{
in }\Omega \text{,} \\ 
\qquad \qquad \qquad \qquad \qquad \qquad \quad \mathbf{u}_{\varepsilon }=0%
\text{ on }\partial \Omega \text{,}%
\end{array}%
\right.
\end{equation*}%
with $\mathbf{f}=\left( f^{i}\right) \in H^{-1}\left( \Omega ;\mathbb{R}%
\right) ^{N}$ and (\ref{eq1.3b}), by the well-known approach of \textit{%
reiterated two-scale convergence}. Further, the homogenization of the steady
fluids flows has been investigated by many authors in various directions.
Let us bring up the work by M. Sango and J.L. Woukeng on the homogenization
of steady rotating fluids in \cite{bib10}. We can also mention the work of
S. Wright in \cite{bib15} and \cite{bib16}.

This paper deals with the reiterated homogenization of the unsteady
nonlinear fluids flows in the deterministic setting. The passage to the more
general behaviours of the viscosites (\textit{beyong the periodic setting})
is a simple adaptation of \cite{bib8}. As far as i know, this work is not
yet present in the literature of the Navier-Stokes type flows and
homogenization. Our approach is the reiterated two-scale convergence ideas 
\cite{bib5}, \cite{bib6}, \cite{bib8}, and the derived macroscopic
homogenized equations are of the Navier-Stokes type.

Unless otherwise specified, vector spaces throughout are considered over the
complex field, $\mathbb{C}$, and scalar functions are assumed to take
complex values. Let us recall some basic notation. If $X$ and $F$ denote a
locally compact space and a Banach space, respectively, then we write $%
\mathcal{C}\left( X;F\right) $ for continuous mappings of $X$ into $F$, and $%
\mathcal{B}\left( X;F\right) $ for those mappings in $\mathcal{C}\left(
X;F\right) $ that are bounded. We denote by $\mathcal{K}\left( X;F\right) $
the mappings in $\mathcal{C}\left( X;F\right) $ having compact supports.\ We
shall assume $\mathcal{B}\left( X;F\right) $ to be equipped with the
supremum norm $\left\Vert u\right\Vert _{\infty }=\sup_{x\in X}\left\Vert
u\left( x\right) \right\Vert $ ($\left\Vert {\small \cdot }\right\Vert $
denotes the norm in $F$). For shortness we will write $\mathcal{C}\left(
X\right) =\mathcal{C}\left( X;\mathbb{C}\right) $, $\mathcal{B}\left(
X\right) =\mathcal{B}\left( X;\mathbb{C}\right) $ and $\mathcal{K}\left(
X\right) =\mathcal{K}\left( X;\mathbb{C}\right) $. Likewise in the case when 
$F=\mathbb{C}$, the usual spaces $L^{p}\left( X;F\right) $ and $%
L_{loc}^{p}\left( X;F\right) $ ($X$ provided with a positive Radon measure)
will be denoted by $L^{p}\left( X\right) $ and $L_{loc}^{p}\left( X\right) $%
, respectively. Finally, the numerical space $\mathbb{R}^{N}$ and its open
sets are each provided with Lebesgue measure denoted by $dx=dx_{1}...dx_{N}$.

The rest of the paper is organized as follows: In Section 2 we recall the
concept of reiterated \textit{two-scale convergence }with all its
fundamental results. The Section 3 is devoted to the homogenization of
problem (\ref{eq1.2a})-(\ref{eq1.2d}).

\section{Preliminaries}

Let us first recall some notions of traces of functions. Let $B$ be a closed
vector subspace of $\mathcal{B}\left( \mathbb{R}_{y}^{N}\times \mathbb{R}%
_{\tau }\times \mathbb{R}_{z}^{N}\right) $ (where $\mathbb{R}_{\tau }$ is
the numerical space $\mathbb{R}$ of variable $\tau $ ). For any real number $%
\varepsilon >0$ and for any $u\in C(\overline{Q};B)$ we set 
\begin{equation}
u^{\varepsilon }\left( x,t\right) =u\left( x,t,\frac{x}{\varepsilon },\frac{t%
}{\varepsilon },\frac{x}{\varepsilon ^{2}}\right) \text{ for all }\left(
x,t\right) \in \overline{Q}\text{.}  \label{eq2.1a}
\end{equation}%
Then (\ref{eq2.1a}) defines a function $u^{\varepsilon }\in C\left( 
\overline{Q}\right) $ such that $\left\vert u^{\varepsilon }\left(
x,t\right) \right\vert \leq \left\Vert u\left( x,t\right) \right\Vert
_{\infty }$ for all $\left( x,t\right) \in \overline{Q}$. Further, one has $%
\left\Vert u^{\varepsilon }\right\Vert _{L^{p}\left( Q\right) }\leq
\left\Vert u\right\Vert _{L^{p}\left( Q;B\right) }$ $\left( 1\leq p\leq
\infty \right) $. Thus, we have a linear operator $u\rightarrow
u^{\varepsilon }$ of $C(\overline{Q};B)$ into $L^{p}\left( Q\right) $ which
is continuous for the $L^{p}\left( Q;B\right) $-norm. Hence, in virtue of
the density of $C(\overline{Q};B)$ in $L^{p}\left( Q;B\right) $ one has the
following proposition.

\begin{proposition}
\label{pr2.1} The operator $u\rightarrow u^{\varepsilon }$ of $C(\overline{Q}%
;B)$ into $L^{p}\left( Q\right) $ is extended \ by continuity to a unique
linear operator of $L^{p}\left( Q;B\right) $ into $L^{p}\left( Q\right) $
still denoted by $u\rightarrow u^{\varepsilon }$ such that $\left\Vert
u^{\varepsilon }\right\Vert _{L^{p}\left( Q\right) }$ $\leq \left\Vert
u\right\Vert _{L^{p}\left( Q;B\right) }$.
\end{proposition}

\noindent This gives a sense to (\ref{eq2.1a}) for $u\in L^{p}\left(
Q;B\right) $\quad $\left( 1\leq p\leq \infty \right) $. Let us define now $%
u^{\varepsilon }$ of the form (\ref{eq2.1a}) for $u\in C\left( \overline{Q}%
;L^{\infty }\left( \mathbb{R}_{y}^{N}\times \mathbb{R}_{\tau };\mathcal{B}%
\left( \mathbb{R}_{z}^{N}\right) \right) \right) $. We consider first a
function $\psi \in L^{\infty }\left( \mathbb{R}_{y}^{N}\times \mathbb{R}%
_{\tau };\mathcal{B}\left( \mathbb{R}_{z}^{N}\right) \right) $ and we put $%
^{\varepsilon }\psi \left( y,\tau \right) =\psi \left( y,\tau ,\frac{y}{%
\varepsilon }\right) $ for $\left( y,\tau \right) \in \mathbb{R}^{N}\times 
\mathbb{R}$ and $\varepsilon >0$. This gives the function $^{\varepsilon
}\psi \in L^{\infty }\left( \mathbb{R}_{y}^{N}\times \mathbb{R}_{\tau
}\right) $ with $\left\Vert ^{\varepsilon }\psi \right\Vert _{L^{\infty
}\left( \mathbb{R}_{y}^{N}\times \mathbb{R}_{\tau }\right) }\leq \left\Vert
\psi \right\Vert _{L^{\infty }\left( \mathbb{R}_{y}^{N}\times \mathbb{R}%
_{\tau };\mathcal{B}\left( \mathbb{R}_{z}^{N}\right) \right) }$. Next, we
define 
\begin{equation}
\psi ^{\varepsilon }\left( x,t\right) \mathfrak{=}^{\varepsilon }\psi \left( 
\frac{x}{\varepsilon },\frac{t}{\varepsilon }\right) =\psi \left( \frac{x}{%
\varepsilon },\frac{t}{\varepsilon },\frac{x}{\varepsilon ^{2}}\right) \text{%
\quad }\left( \left( x,t\right) \in \mathbb{R}^{N}\times \mathbb{R}\right) 
\text{.}  \label{eq2.1b}
\end{equation}%
We have $\psi ^{\varepsilon }\in L^{\infty }\left( \mathbb{R}^{N}\times 
\mathbb{R}\right) $ and $\left\Vert \psi ^{\varepsilon }\right\Vert
_{L^{\infty }\left( \mathbb{R}^{N}\times \mathbb{R}\right) }\leq \left\Vert
\psi \right\Vert _{L^{\infty }\left( \mathbb{R}_{y}^{N}\times \mathbb{R}%
_{\tau };\mathcal{B}\left( \mathbb{R}_{z}^{N}\right) \right) }$, of course.
Let us denote by $C\left( \overline{Q}\right) \otimes L^{\infty }\left( 
\mathbb{R}_{y}^{N}\times \mathbb{R}_{\tau };\mathcal{B}\left( \mathbb{R}%
_{z}^{N}\right) \right) $ the space of functions $u$ of the form 
\begin{equation*}
u=\sum_{i\in I}\varphi _{i}\otimes \psi _{i}\text{ with }\varphi _{i}\in
C\left( \overline{Q}\right) \text{ and }\psi _{i}\in L^{\infty }\left( 
\mathbb{R}_{y}^{N}\times \mathbb{R}_{\tau };\mathcal{B}\left( \mathbb{R}%
_{z}^{N}\right) \right) \text{,}
\end{equation*}%
where $I$ is a finite set of indices. For any $u\in C\left( \overline{Q}%
\right) \otimes L^{\infty }\left( \mathbb{R}_{y}^{N}\times \mathbb{R}_{\tau
};\mathcal{B}\left( \mathbb{R}_{z}^{N}\right) \right) $, we put 
\begin{equation}
u^{\varepsilon }\left( x,t\right) =\sum_{i\in I}\varphi _{i}\left(
x,t\right) \psi _{i}^{\varepsilon }\left( x,t\right) \text{ for }\left(
x,t\right) \in \overline{Q}  \label{eq2.1c}
\end{equation}%
where $\psi _{i}^{\varepsilon }$ is defined as in (\ref{eq2.1b}). One has $%
u^{\varepsilon }\in L^{\infty }\left( Q\right) $ with

\noindent $\left\Vert u^{\varepsilon }\right\Vert _{L^{\infty }\left(
Q\right) }\leq \sup_{\left( x,t\right) \in \overline{Q}}\left\Vert u\left(
x,t\right) \right\Vert _{L^{\infty }\left( \mathbb{R}_{y}^{N}\times \mathbb{R%
}_{\tau };\mathcal{B}\left( \mathbb{R}_{z}^{N}\right) \right) }$. In virtue
of the density of $C\left( \overline{Q}\right) \otimes L^{\infty }\left( 
\mathbb{R}_{y}^{N}\times \mathbb{R}_{\tau };\mathcal{B}\left( \mathbb{R}%
_{z}^{N}\right) \right) $ in $C\left( \overline{Q};L^{\infty }\left( \mathbb{%
R}_{y}^{N}\times \mathbb{R}_{\tau };\mathcal{B}\left( \mathbb{R}%
_{z}^{N}\right) \right) \right) $ (see for e.g. \cite[page 46]{bib3}), It
follows that

\begin{proposition}
\label{pr2.2} The operator $u\rightarrow u^{\varepsilon }$ of $C\left( 
\overline{Q}\right) \otimes L^{\infty }\left( \mathbb{R}_{y}^{N}\times 
\mathbb{R}_{\tau };\mathcal{B}\left( \mathbb{R}_{z}^{N}\right) \right) $
into $L^{\infty }\left( Q\right) $ ($u^{\varepsilon }$ being defined as in (%
\ref{eq2.1c})) is extended by continuity to a linear continuous operator
still denoted by $u\rightarrow u^{\varepsilon }$, of $C\left( \overline{Q}%
;L^{\infty }\left( \mathbb{R}_{y}^{N}\times \mathbb{R}_{\tau };\mathcal{B}%
\left( \mathbb{R}_{z}^{N}\right) \right) \right) $ into $L^{\infty }\left(
Q\right) $.
\end{proposition}

Let us recall the following preliminaries. We assume once for all that $N=2$%
. We set $Y=\left( -\frac{1}{2},\frac{1}{2}\right) ^{N}$ , $Z=\left( -\frac{1%
}{2},\frac{1}{2}\right) ^{N}$ and $\mathcal{T=}\left( -\frac{1}{2},\frac{1}{2%
}\right) $, $Y$, $Z$ and $\mathcal{T}$ being considered as subsets of $%
\mathbb{R}_{y}^{N}$ (the space $\mathbb{R}^{N}$ of variables $y=\left(
y_{1},...,y_{N}\right) $), $\mathbb{R}_{z}^{N}$ (the space $\mathbb{R}^{N}$
of variables $z=\left( z_{1},...,z_{N}\right) $) and $\mathbb{R}_{\tau }$
(the space $\mathbb{R}$ of variable $\tau $), respectively. Our purpose is
to study the homogenization of (\ref{eq1.2a})-(\ref{eq1.2d}) under the
periodicity hypothesis on $a_{ij}$.

Let us first recall that a function $u\in L_{loc}^{1}\left( \mathbb{R}%
_{y}^{N}\times \mathbb{R}_{\tau }\times \mathbb{R}_{z}^{N}\right) $ is said
to be $Y\times \mathcal{T}\times Z$-periodic if for each $\left(
k,r,l\right) \in \mathbb{Z}^{N}\times \mathbb{Z\times Z}^{N}$ ($\mathbb{Z}$
denotes the integers), we have $u\left( y+k,\tau +r,z+l\right) =u\left(
y,\tau ,z\right) $ almost everywhere (a.e.) in $\left( y,\tau ,z\right) \in 
\mathbb{R}^{N}\times \mathbb{R\times R}^{N}$. The space of all $Y\times 
\mathcal{T}\times Z$-periodic continuous complex functions on $\mathbb{R}%
_{y}^{N}\times \mathbb{R}_{\tau }\times \mathbb{R}_{z}^{N}$ is denoted by $%
\mathcal{C}_{per}\left( Y\times \mathcal{T}\times Z\right) $; that of all $%
Y\times \mathcal{T}\times Z$-periodic functions in $L_{loc}^{p}\left( 
\mathbb{R}_{y}^{N}\times \mathbb{R}_{\tau }\times \mathbb{R}_{z}^{N}\right) $
$\left( 1\leq p<\infty \right) $ is denoted by $L_{per}^{p}\left( Y\times 
\mathcal{T}\times Z\right) $. $\mathcal{C}_{per}\left( Y\times \mathcal{T}%
\times Z\right) $ is a Banach space under the supremum norm on $\mathbb{R}%
^{N}\times \mathbb{R\times R}^{N}$, whereas $L_{per}^{p}\left( Y\times 
\mathcal{T}\times Z\right) $ is a Banach space under the norm 
\begin{equation*}
\left\Vert u\right\Vert _{L^{p}\left( Y\times Z\right) }=\left( \int \int
\int_{Y\times \mathcal{T}\times Z}\left\vert u\left( y,z\right) \right\vert
^{p}dydzd\tau \right) ^{\frac{1}{p}}\text{ }\left( u\in L_{per}^{p}\left(
Y\times \mathcal{T}\times Z\right) \right) \text{.}
\end{equation*}

We will need the space $H_{\#}^{1}\left( Y\right) $ (resp. $H_{\#}^{1}\left(
Z\right) $) of $Y$-periodic (resp. $Z$-periodic) functions $u\in
H_{loc}^{1}\left( \mathbb{R}_{y}^{N}\right) =W_{loc}^{1,2}\left( \mathbb{R}%
_{y}^{N}\right) $ (resp. $u\in H_{loc}^{1}\left( \mathbb{R}_{z}^{N}\right)
=W_{loc}^{1,2}\left( \mathbb{R}_{z}^{N}\right) $) such that $\int_{Y}u\left(
y\right) dy=0$ (resp. $\int_{Z}u\left( z\right) dz=0$). Provided with the
gradient norm, 
\begin{equation*}
\left\Vert u\right\Vert _{H_{\#}^{1}\left( Y\right) }=\left(
\int_{Y}\left\vert \nabla _{y}u\right\vert ^{2}dy\right) ^{\frac{1}{2}}\text{
}\left( u\in H_{\#}^{1}\left( Y\right) \right)
\end{equation*}%
\begin{equation*}
\text{(resp. }\left\Vert u\right\Vert _{H_{\#}^{1}\left( Z\right) }=\left(
\int_{Z}\left\vert \nabla _{z}u\right\vert ^{2}dz\right) ^{\frac{1}{2}}\text{
}\left( u\in H_{\#}^{1}\left( Z\right) \right) \text{),}
\end{equation*}%
where $\nabla _{y}u=\left( \frac{\partial u}{\partial y_{1}},...,\frac{%
\partial u}{\partial y_{N}}\right) $ (resp. $\nabla _{z}u=\left( \frac{%
\partial u}{\partial z_{1}},...,\frac{\partial u}{\partial z_{N}}\right) $), 
$H_{\#}^{1}\left( Y\right) $ (resp. $H_{\#}^{1}\left( Z\right) $) is a
Hilbert space.

Before we can recall the concept of reiterated two-scale convergence, let us
introduce one further notation. The letter $E$ throughout will denote a
family of real numbers $0<\varepsilon <1$ admitting $0$ as an accumulation
point. For example, $E$ may be the whole interval $\left( 0,1\right) $; $E$
may also be an ordinary sequence $\left( \varepsilon _{n}\right) _{n\in 
\mathbb{N}}$ with $0<\varepsilon _{n}<1$ and $\varepsilon _{n}\rightarrow 0$
as $n\rightarrow \infty $. In the latter case $E$ will be referred to as a 
\textit{fundamental sequence}. Let us also recall the following fundamental
results.

\begin{proposition}
\label{pr2.3} For any $u\in L^{p}\left( Q;\mathcal{C}_{per}\left( Y\times 
\mathcal{T}\times Z\right) \right) $ $\left( 1\leq p<+\infty \right) $, we
consider the complex function $\widetilde{u}$ on $Q$ defined by $\widetilde{u%
}\left( x,t\right) =\int \int \int_{Y\times \mathcal{T}\times Z}u\left(
x,y,\tau ,z\right) dydzd\tau $\quad $\left( \left( x,t\right) \in Q\right) $%
. As $\varepsilon \rightarrow 0$, we have $u^{\varepsilon }\rightarrow 
\widetilde{u}$ in $L^{\infty }\left( Q\right) $-weak $\ast $ for $u\in 
\mathcal{C}\left( \overline{Q};\mathcal{C}_{per}\left( Y\times \mathcal{T}%
\times Z\right) \right) $ and $u^{\varepsilon }\rightarrow \widetilde{u}$ in 
$L^{p}\left( Q\right) $-weak for $u\in L^{p}\left( Q;\mathcal{C}_{per}\left(
Y\times \mathcal{T}\times Z\right) \right) $, where $u^{\varepsilon }$ is
defined as in (\ref{eq2.1a}).
\end{proposition}

\begin{proof}
It is a simple adaptation of the proof in \cite[Proposition 1.9 and
Proposition 1.10]{bib9}.
\end{proof}

\begin{definition}
\label{def2.1} A sequence $\left( u_{\varepsilon }\right) _{\varepsilon \in
E}\subset L^{p}\left( Q\right) $ $\left( 1\leq p<\infty \right) $ is said to:

(i) weakly two-scale converge reiteratively in $L^{p}\left( Q\right) $ to
some $u_{0}\in L^{p}\left( Q;L_{per}^{p}\left( Y\times \mathcal{T}\times
Z\right) \right) $ if as

\noindent $E\ni \varepsilon \rightarrow 0$, 
\begin{equation}
\int_{Q}u_{\varepsilon }\left( x,t\right) \psi ^{\varepsilon }\left(
x,t\right) dx\rightarrow \int \int \int \int_{Q\times Y\times \mathcal{T}%
\times Z}u_{0}\left( x,y,\tau ,z\right) \psi \left( x,y,\tau ,z\right)
dxdtdydzd\tau  \label{eq2.1}
\end{equation}%
\begin{equation*}
\begin{array}{c}
\text{for all }\psi \in L^{p^{\prime }}\left( Q;\mathcal{C}_{per}\left(
Y\times \mathcal{T}\times Z\right) \right) \text{ }\left( \frac{1}{p^{\prime
}}=1-\frac{1}{p}\right) \text{, where }\psi ^{\varepsilon }\left( x,t\right)
= \\ 
\psi \left( x,t,\frac{x}{\varepsilon },\frac{t}{\varepsilon },\frac{x}{%
\varepsilon ^{2}}\right) \text{ }\left( \left( x,t\right) \in Q\right) \text{%
;}%
\end{array}%
\end{equation*}

(ii) strongly two-scale converge reiteratively in $L^{p}\left( Q\right) $ to
some $u_{0}\in L^{p}\left( Q;L_{per}^{p}\left( Y\times \mathcal{T}\times
Z\right) \right) $ if the following property is verified: 
\begin{equation*}
\left\{ 
\begin{array}{c}
\text{Given }\eta >0\text{ and }v\in L^{p}\left( Q;\mathcal{C}_{per}\left(
Y\times \mathcal{T}\times Z\right) \right) \text{ with} \\ 
\left\Vert u_{0}-v\right\Vert _{L^{p}\left( Q\times Y\times \mathcal{T}%
\times Z\right) }\leq \frac{\eta }{2}\text{, there is some }\alpha >0\text{
such} \\ 
\text{that }\left\Vert u_{\varepsilon }-v^{\varepsilon }\right\Vert
_{L^{p}\left( Q\right) }\leq \eta \text{ provided }E\ni \varepsilon \leq
\alpha \text{,}%
\end{array}%
\right.
\end{equation*}%
where $v^{\varepsilon }\left( x,t\right) =v\left( x,t,\frac{x}{\varepsilon },%
\frac{t}{\varepsilon },\frac{x}{\varepsilon ^{2}}\right) $ $\left( \left(
x,t\right) \in Q\right) $.
\end{definition}

We will express this by writing $u_{\varepsilon }\rightarrow u_{0}$ \textit{%
reiteratively in} $L^{p}\left( Q\right) $-\textit{weak} $\Sigma $ \ for the
case in (i), and $u_{\varepsilon }\rightarrow u_{0}$ \textit{reiteratively in%
} $L^{p}\left( Q\right) $-\textit{strong} $\Sigma $ in case (ii).

\begin{example}
\label{ex2.1a} For any $u\in L^{p}\left( Q;\mathcal{C}_{per}\left( Y\times 
\mathcal{T}\times Z\right) \right) $, 
\begin{equation}
u^{\varepsilon }\rightarrow u\text{ reiteratively in }L^{p}\left( Q\right) 
\text{-weak }\Sigma \text{,}  \label{eq2.2}
\end{equation}%
as $\varepsilon \rightarrow 0$ (use Proposition \ref{pr2.3}). In particular,
(\ref{eq2.2}) is satisfied for $u\in \mathcal{C}_{per}\left( Y\times 
\mathcal{T}\times Z\right) $ as $\varepsilon \rightarrow 0$.
\end{example}

\begin{example}
\label{ex2.1b} (i) Suppose $u_{0}\in L^{p}\left( Q;\mathcal{C}_{per}\left(
Y\times \mathcal{T}\times Z\right) \right) $. Then $u_{\varepsilon
}\rightarrow u_{0}$ reiteratively in $L^{p}\left( Q\right) $-\textit{strong} 
$\Sigma $ if and only if $\left\Vert u_{\varepsilon }-u_{0}^{\varepsilon
}\right\Vert _{L^{p}\left( Q\right) }\rightarrow 0$ as $E\ni \varepsilon
\rightarrow 0$.

(ii) For any $u\in L^{p}\left( Q;\mathcal{C}_{per}\left( Y\times \mathcal{T}%
\times Z\right) \right) $, $u^{\varepsilon }\rightarrow u$ reiteratively in $%
L^{p}\left( Q\right) $-strong $\Sigma $.

(iii) If $u_{\varepsilon }\rightarrow u_{0}$ in $L^{p}\left( Q\right) $ as $%
E\ni \varepsilon \rightarrow 0$, then $u_{\varepsilon }\rightarrow u_{0}$
reiteratively in $L^{p}\left( Q\right) $-\textit{strong} $\Sigma $.
\end{example}

\begin{proposition}
\label{pr2.4} Suppose that a sequence $\left( u_{\varepsilon }\right)
_{\varepsilon \in E}\subset L^{p}\left( Q\right) $ weakly two-scale
converges reiteratively in $L^{p}\left( Q\right) $ to some $u_{0}\in
L^{p}\left( Q;L_{per}^{p}\left( Y\times \mathcal{T}\times Z\right) \right) $%
. Let $u_{0}^{\#}\in L^{p}\left( Q;L_{per}^{p}\left( Y\times \mathcal{T}%
\right) \right) $ be defined by $u_{0}^{\#}\left( x,t,y,\tau \right)
=\int_{Z}u_{0}\left( x,t,y,\tau ,z\right) dz$ \ for almost all $\left(
x,t\right) \in Q$ and almost all $\left( y,\tau \right) \in Y\times \mathcal{%
T}$, and let $\widetilde{u}_{0}\in L^{p}\left( Q\right) $ be defined by $%
\widetilde{u}_{0}\left( x,t\right) =\int \int \int_{Y\times \mathcal{T}%
\times Z}u_{0}\left( x,t,y,\tau ,z\right) dydzd\tau $ $\ \left( \left(
x,t\right) \in Q\right) $. Then,

(i) $\left( u_{\varepsilon }\right) _{\varepsilon \in E}$ two-scale
converges in $L^{p}\left( Q\right) $ to $u_{0}^{\#}$, i.e., as $E\ni
\varepsilon \rightarrow 0$,%
\begin{equation*}
\int_{Q}u_{\varepsilon }\left( x,t\right) f^{\varepsilon }\left( x,t\right)
dxdt\rightarrow \int \int \int_{Q\times Y\times \mathcal{T}}u_{0}^{\#}\left(
x,t,y,\tau \right) f\left( x,t,y,\tau \right) dxdtdyd\tau
\end{equation*}%
for all $f\in L^{p^{\prime }}\left( Q;\mathcal{C}_{per}\left( Y\times 
\mathcal{T}\right) \right) $ (with $f^{\varepsilon }\left( x,t\right)
=f\left( x,t,\frac{x}{\varepsilon },\frac{t}{\varepsilon }\right) $\quad $%
\left( \left( x,t\right) \in Q\right) $);

(ii) $\left( u_{\varepsilon }\right) _{\varepsilon \in E}$ weakly converges
in $L^{p}\left( Q\right) $ to $\widetilde{u}_{0}$.
\end{proposition}

\begin{proof}
Let us notice that $L^{p^{\prime }}\left( Q;\mathcal{C}_{per}\left( Y\times 
\mathcal{T}\right) \right) \otimes \mathcal{C}_{per}\left( Z\right) $ and $%
L^{p^{\prime }}\left( Q\right) \otimes \mathcal{C}_{per}\left( Y\times 
\mathcal{T}\times Z\right) $ are the subsets of $L^{p^{\prime }}\left( Q;%
\mathcal{C}_{per}\left( Y\times \mathcal{T}\times Z\right) \right) $. Thus,
on one hand 
\begin{equation*}
\int_{Q}u_{\varepsilon }\left( x,t\right) f\left( x,t,\frac{x}{\varepsilon },%
\frac{t}{\varepsilon }\right) dxdt\rightarrow
\end{equation*}%
\begin{equation*}
\int \int \int_{Q\times Y\times \mathcal{T}}\left( \int_{Z}u_{0}\left(
x,t,y,\tau ,z\right) dz\right) f\left( x,t,y,\tau \right) dxdtdyd\tau
\end{equation*}%
for all $f\in L^{p^{\prime }}\left( Q;\mathcal{C}_{per}\left( Y\times 
\mathcal{T}\right) \right) $, and on the other hand%
\begin{equation*}
\int_{Q}u_{\varepsilon }\left( x,t\right) f\left( x,t\right) dxdt\rightarrow
\int_{Q}\left( \int \int_{Y\times \mathcal{T}\times Z}u_{0}\left( x,t,y,\tau
,z\right) dydz\right) f\left( x,t\right) dxdt
\end{equation*}%
for all $f\in L^{p^{\prime }}\left( Q\right) $ as $E\ni \varepsilon
\rightarrow 0$. Hence, (i) and (ii) follow.
\end{proof}

\begin{remark}
\label{rem2.1} It is of interest to know that if $u_{\varepsilon
}\rightarrow u_{0}$ reiteratively in $L^{p}\left( Q\right) $-weak $\Sigma $,
then (\ref{eq2.1}) holds for $\psi \in \mathcal{C}\left( \overline{Q}%
;L_{per}^{\infty }\left( Y\times \mathcal{T};\mathcal{C}_{per}\left(
Z\right) \right) \right) $. The proof is analogous to the one in \cite[%
Proposition 3.3]{bib9}. Further, as a consequence of Proposition \ref{pr2.4}%
, if $u\in L^{p}\left( Q;\mathcal{C}_{per}\left( Y\times \mathcal{T}\times
Z\right) \right) $ then $u^{\varepsilon }\rightarrow u$\ reiteratively in $%
L^{p}\left( Q\right) $-strong $\Sigma $ and\emph{\ }%
\begin{equation}
\lim_{\varepsilon \rightarrow 0}\left\Vert u^{\varepsilon }\right\Vert
_{L^{p}\left( Q\right) }=\left\Vert u\right\Vert _{L^{p}\left(
Q;L_{per}^{p}\left( Y\times \mathcal{T}\times Z\right) \right) }\text{.}
\label{eq2.3}
\end{equation}
\end{remark}

Let us state the following important propositions.

\begin{proposition}
\label{pr2.5} Suppose a sequence $\left( u_{\varepsilon }\right)
_{\varepsilon \in E}$ strongly two-scale converges reiteratively in $%
L^{p}\left( Q\right) $ $\left( 1\leq p<\infty \right) $ to some $u_{0}\in
L^{p}\left( Q;L_{per}^{p}\left( Y\times \mathcal{T}\times Z\right) \right) $%
. Then, as $E\ni \varepsilon \rightarrow 0$,

(i) $u_{\varepsilon }\rightarrow u_{0}$ reiteratively in $L^{p}\left(
Q\right) $-weak $\Sigma $

(ii) $\left\Vert u_{\varepsilon }\right\Vert _{L^{p}\left( Q\right)
}\rightarrow \left\Vert u_{0}\right\Vert _{L^{p}\left( Q;L_{per}^{p}\left(
Y\times \mathcal{T}\times Z\right) \right) }$.

Conversely, if $p=2$ and if assertions (i)-(ii) are satisfied, then $%
u_{\varepsilon }\rightarrow u_{0}$ reiteratively in $L^{p}\left( Q\right) $%
-strong $\Sigma $.
\end{proposition}

\begin{proposition}
\label{pr2.6} Suppose $1\leq p,q<+\infty $ are such that $\frac{1}{r}=\frac{1%
}{p}+\frac{1}{q}\leq 1$. Let $u_{0}\in L^{p}\left( Q;L_{per}^{p}\left(
Y\times \mathcal{T}\times Z\right) \right) $, $v_{0}\in L^{q}\left(
Q;L_{per}^{q}\left( Y\times \mathcal{T}\times Z\right) \right) $, $%
u_{\varepsilon }\in L^{p}\left( Q\right) $ and $v_{\varepsilon }\in
L^{q}\left( Q\right) $ with $\varepsilon \in E$. If $u_{\varepsilon
}\rightarrow u_{0}$ reiteratively in $L^{p}\left( Q\right) $-strong $\Sigma $
and $v_{\varepsilon }\rightarrow v_{0}$ reiteratively in $L^{q}\left(
Q\right) $-weak $\Sigma $, then $u_{\varepsilon }v_{\varepsilon }\rightarrow
u_{0}v_{0}$ reiteratively in $L^{r}\left( Q\right) $-weak $\Sigma $.
\end{proposition}

For more details on the reiterated two-scale convergence for periodic
structures, we find it more convenient to draw the reader's attention to a
few references, e.g., \cite{bib5}, \cite{bib6} and \cite{bib8}. However, we
recall below two fundamental results. First of all, let 
\begin{equation*}
\mathcal{Y}\left( 0,T\right) =\left\{ v\in L^{2}\left( 0,T;H_{0}^{1}\left(
\Omega ;\mathbb{R}\right) \right) :v^{\prime }\in L^{2}\left(
0,T;H^{-1}\left( \Omega ;\mathbb{R}\right) \right) \right\} \text{.}
\end{equation*}%
$\mathcal{Y}\left( 0,T\right) $ is provided with the norm 
\begin{equation*}
\left\Vert v\right\Vert _{\mathcal{Y}\left( 0,T\right) }=\left( \left\Vert
v\right\Vert _{L^{2}\left( 0,T;H_{0}^{1}\left( \Omega \right) \right)
}^{2}+\left\Vert v^{\prime }\right\Vert _{L^{2}\left( 0,T;H^{-1}\left(
\Omega \right) \right) }^{2}\right) ^{\frac{1}{2}}\qquad \left( v\in 
\mathcal{Y}\left( 0,T\right) \right)
\end{equation*}%
which makes it a Hilbert space.

\begin{theorem}
\label{th2.1} Assume that $1<p<\infty $ and further $E$ is a fundamental
sequence. Let a sequence $\left( u_{\varepsilon }\right) _{\varepsilon \in
E} $ be bounded in $L^{p}\left( Q\right) $. Then, a subsequence $E^{\prime }$
can be extracted from $E$ such that $\left( u_{\varepsilon }\right)
_{\varepsilon \in E^{\prime }}$ weakly two-scale converges reiteratively in $%
L^{p}\left( Q\right) $.
\end{theorem}

\begin{theorem}
\label{th2.2} Let $E$ be a fundamental sequence. Suppose a sequence $\left(
u_{\varepsilon }\right) _{\varepsilon \in E}$ is bounded in $\mathcal{Y}%
\left( 0,T\right) $. Then, a subsequence $E^{\prime }$ can be extracted from 
$E$ such that, as $E^{\prime }\ni \varepsilon \rightarrow 0$, 
\begin{equation*}
u_{\varepsilon }\rightarrow u_{0}\text{ in }\mathcal{Y}\left( 0,T\right) 
\text{-weak,\qquad \qquad \qquad \qquad \qquad \qquad \qquad }
\end{equation*}%
\begin{equation*}
u_{\varepsilon }\rightarrow u_{0}\text{ reiteratively in }L^{2}\left(
Q\right) \text{-weak }\Sigma \text{,\qquad \qquad \qquad \qquad }
\end{equation*}%
\begin{equation*}
\frac{\partial u_{\varepsilon }}{\partial x_{j}}\rightarrow \frac{\partial
u_{0}}{\partial x_{j}}+\frac{\partial u_{1}}{\partial y_{j}}+\frac{\partial
u_{2}}{\partial z_{j}}\text{ reiteratively in }L^{2}\left( Q\right) \text{%
-weak }\Sigma \text{ }\left( 1\leq j\leq N\right) \text{,}
\end{equation*}%
where $u_{0}\in \mathcal{Y}\left( 0,T\right) $, $u_{1}\in L^{2}\left(
Q;L_{per}^{2}\left( \mathcal{T};H_{\#}^{1}\left( Y\right) \right) \right) $
and $u_{2}\in L^{2}\left( Q;L_{per}^{2}\left( \mathcal{T};L_{per}^{2}\left(
Y;H_{\#}^{1}\left( Z\right) \right) \right) \right) $.
\end{theorem}

The proofs of Theorem \ref{th2.1} and Theorem \ref{th2.2} can be found in,
e.g., \cite{bib8} and \cite{bib13}, respectively.

\begin{theorem}
\label{th2.3} Let $\left( u_{\varepsilon }\right) _{\varepsilon >0}$ be a
sequence in $L^{2}\left( 0,T;H^{1}\left( \Omega \right) \right) $ such that
there are three functions $u_{0}\in L^{2}\left( 0,T;H^{1}\left( \Omega
\right) \right) $, $u_{1}\in L^{2}\left( Q;L_{per}^{2}\left( \mathcal{T}%
;H_{\#}^{1}\left( Y\right) \right) \right) $ and $u_{2}\in L^{2}\left(
Q;L_{per}^{2}\left( \mathcal{T};L_{per}^{2}\left( Y;H_{\#}^{1}\left(
Z\right) \right) \right) \right) $ verifying 
\begin{equation}
\frac{\partial u_{\varepsilon }}{\partial x_{j}}\rightarrow \frac{\partial
u_{0}}{\partial x_{j}}+\frac{\partial u_{1}}{\partial y_{j}}+\frac{\partial
u_{2}}{\partial z_{j}}\text{ reiteratively in }L^{2}\left( Q\right) \text{%
-strong }\Sigma \text{ }\left( 1\leq j\leq N\right)  \label{eq2.4}
\end{equation}%
as $\varepsilon \rightarrow 0$. Let $\eta >0$. By the density of $%
L^{2}\left( 0,T;H^{1}\left( \Omega \right) \right) \otimes \mathcal{C}%
_{per}\left( \mathcal{T};\mathcal{C}_{per}^{1}\left( Y\right) \mathfrak{/}%
\mathbb{C}\right) $ in $L^{2}\left( Q;L_{per}^{2}\left( \mathcal{T}%
;H_{\#}^{1}\left( Y\right) \right) \right) $ ($\mathcal{C}_{per}^{1}\left(
Y\right) \mathfrak{/}\mathbb{C=}\left\{ w\in \mathcal{C}_{per}^{1}\left(
Y\right) :\int_{Y}w\left( y\right) dy=0\right\} $ and $\mathcal{C}%
_{per}^{1}\left( Y\right) =\mathcal{C}_{per}\left( Y\right) \cap \mathcal{C}%
^{1}\left( \mathbb{R}_{y}^{N}\right) $) on one hand, and that of $%
L^{2}\left( 0,T;H^{1}\left( \Omega \right) \right) \otimes \mathcal{C}%
_{per}\left( \mathcal{T};\mathcal{C}_{per}^{1}\left( Y;\mathcal{C}%
_{per}^{1}\left( Z\right) \mathfrak{/}\mathbb{C}\right) \right) $

in $L^{2}\left( Q;L_{per}^{2}\left( \mathcal{T};L_{per}^{2}\left(
Y;H_{\#}^{1}\left( Z\right) \right) \right) \right) $ on the other hand ($%
\mathcal{C}_{per}^{1}\left( Z\right) \mathfrak{/}\mathbb{C=}\left\{ w\in 
\mathcal{C}_{per}^{1}\left( Z\right) :\int_{Z}w\left( z\right) dz=0\right\} $
and $\mathcal{C}_{per}^{1}\left( Z\right) =\mathcal{C}_{per}\left( Z\right)
\cap \mathcal{C}^{1}\left( \mathbb{R}_{y}^{N}\right) $), fix some $\psi
_{1}\in L^{2}\left( 0,T;H^{1}\left( \Omega \right) \right) \otimes \mathcal{C%
}_{per}\left( \mathcal{T};\mathcal{C}_{per}^{1}\left( Y\right) \mathfrak{/}%
\mathbb{C}\right) $ and $\psi _{2}\in L^{2}\left( 0,T;H^{1}\left( \Omega
\right) \right) \otimes \mathcal{C}_{per}\left( \mathcal{T};\mathcal{C}%
_{per}^{1}\left( Y;\mathcal{C}_{per}^{1}\left( Z\right) \mathfrak{/}\mathbb{C%
}\right) \right) $ such that 
\begin{equation*}
\left\Vert u_{1}-\psi _{1}\right\Vert _{L^{2}\left( Q;L_{per}^{2}\left( 
\mathcal{T};H_{\#}^{1}\left( Y\right) \right) \right) }\leq \frac{\eta }{8}%
\text{ and }\left\Vert u_{2}-\psi _{2}\right\Vert _{L^{2}\left(
Q;L_{per}^{2}\left( \mathcal{T};L_{per}^{2}\left( Y;H_{\#}^{1}\left(
Z\right) \right) \right) \right) }\leq \frac{\eta }{8}\text{.}
\end{equation*}%
Then there is some $\varepsilon _{0}>0$ such that 
\begin{equation*}
\left\Vert \frac{\partial }{\partial x_{j}}\left( u_{\varepsilon
}-u_{0}-\varepsilon \psi _{1}^{\varepsilon }-\varepsilon ^{2}\psi
_{2}^{\varepsilon }\right) \right\Vert _{L^{2}\left( Q\right) }\leq \eta 
\text{\qquad }\left( 1\leq j\leq N\right)
\end{equation*}%
for all $0<\varepsilon \leq \varepsilon _{0}$.
\end{theorem}

\begin{proof}
Let $1\leq j\leq N$. We have 
\begin{equation*}
\left\Vert \left( \frac{\partial u_{1}}{\partial y_{j}}+\frac{\partial u_{2}%
}{\partial z_{j}}\right) -\left( \frac{\partial \psi _{1}}{\partial y_{j}}+%
\frac{\partial \psi _{2}}{\partial z_{j}}\right) \right\Vert _{L^{2}\left(
Q\times Y\times \mathcal{T\times }Z\right) }\leq \left\Vert u_{1}-\psi
_{1}\right\Vert _{L^{2}\left( Q;L_{per}^{2}\left( \mathcal{T}%
;H_{\#}^{1}\left( Y\right) \right) \right) }
\end{equation*}%
\begin{equation*}
\quad +\left\Vert u_{2}-\psi _{2}\right\Vert _{L^{2}\left(
Q;L_{per}^{2}\left( \mathcal{T};L_{per}^{2}\left( Y;H_{\#}^{1}\left(
Z\right) \right) \right) \right) }\leq \frac{\eta }{4}\text{.}
\end{equation*}%
Thus, in view of (\ref{eq2.4}) (see also Definition \ref{def2.1}), there is
some $\varepsilon _{1}>0$ such that 
\begin{equation*}
\left\Vert \frac{\partial u_{\varepsilon }}{\partial x_{j}}-\frac{\partial
u_{0}}{\partial x_{j}}-\left( \frac{\partial \psi _{1}}{\partial y_{j}}%
\right) ^{\varepsilon }-\left( \frac{\partial \psi _{2}}{\partial z_{j}}%
\right) ^{\varepsilon }\right\Vert _{L^{2}\left( Q\right) }\leq \frac{\eta }{%
2}
\end{equation*}%
for all $0<\varepsilon \leq \varepsilon _{1}$. But $\left( \frac{\partial
\psi _{1}}{\partial y_{j}}\right) ^{\varepsilon }=\varepsilon \frac{\partial
\psi _{1}^{\varepsilon }}{\partial x_{j}}-\varepsilon \left( \frac{\partial
\psi _{1}}{\partial x_{j}}\right) ^{\varepsilon }$ and $\left( \frac{%
\partial \psi _{2}}{\partial z_{j}}\right) ^{\varepsilon }=\varepsilon ^{2}%
\frac{\partial \psi _{2}^{\varepsilon }}{\partial x_{j}}-\varepsilon
^{2}\left( \frac{\partial \psi _{2}}{\partial x_{j}}\right) ^{\varepsilon
}-\varepsilon \left( \frac{\partial \psi _{2}}{\partial y_{j}}\right)
^{\varepsilon }$. Therefore, by the inequalities 
\begin{equation*}
\left\Vert \left( \frac{\partial \psi _{1}}{\partial x_{j}}\right)
^{\varepsilon }\right\Vert _{L^{2}\left( Q\right) }\leq \left\Vert \frac{%
\partial \psi _{1}}{\partial x_{j}}\right\Vert _{L^{2}\left( Q;\mathcal{B}%
\left( \mathbb{R}^{N+1}\right) \right) }\text{, }\left\Vert \left( \frac{%
\partial \psi _{2}}{\partial x_{j}}\right) ^{\varepsilon }\right\Vert
_{L^{2}\left( Q\right) }\leq \left\Vert \frac{\partial \psi _{2}}{\partial
x_{j}}\right\Vert _{L^{2}\left( Q;\mathcal{B}\left( \mathbb{R}^{N}\times 
\mathbb{R}\times \mathbb{R}^{N}\right) \right) }
\end{equation*}%
and 
\begin{equation*}
\left\Vert \left( \frac{\partial \psi _{2}}{\partial y_{j}}\right)
^{\varepsilon }\right\Vert _{L^{2}\left( Q\right) }\leq \left\Vert \frac{%
\partial \psi _{2}}{\partial y_{j}}\right\Vert _{L^{2}\left( Q;\mathcal{B}%
\left( \mathbb{R}^{N}\times \mathbb{R}\times \mathbb{R}^{N}\right) \right) }
\end{equation*}%
one quickly arrives at 
\begin{equation*}
\left\Vert \frac{\partial }{\partial x_{j}}\left( u_{\varepsilon
}-u_{0}-\varepsilon \psi _{1}^{\varepsilon }-\varepsilon ^{2}\psi
_{2}^{\varepsilon }\right) \right\Vert _{L^{2}\left( Q\right) }\leq \frac{%
\eta }{2}+c\left( 2\varepsilon +\varepsilon ^{2}\right) \text{\qquad }\left(
0<\varepsilon \leq \varepsilon _{1}\right) \text{,}
\end{equation*}%
where $c$ is a constant with 
\begin{equation*}
c>\max_{1\leq j\leq N}\left( \left\Vert \frac{\partial \psi _{1}}{\partial
x_{j}}\right\Vert _{L^{2}\left( Q;\mathcal{B}\left( \mathbb{R}^{N}\right)
\right) }+\left\Vert \frac{\partial \psi _{2}}{\partial x_{j}}\right\Vert
_{L^{2}\left( \Omega ;\mathcal{B}\left( \mathbb{R}^{N}\times \mathbb{R}%
\times \mathbb{R}^{N}\right) \right) }+\left\Vert \frac{\partial \psi _{2}}{%
\partial y_{j}}\right\Vert _{L^{2}\left( \Omega ;\mathcal{B}\left( \mathbb{R}%
^{N}\times \mathbb{R}\times \mathbb{R}^{N}\right) \right) }\right) \text{.}
\end{equation*}%
Since $\varepsilon $ goes to zero, we look for $\varepsilon _{0}\leq 1$.
Thus, the theorem follows with $\varepsilon _{0}=\min \left( 1,\varepsilon
_{1},\frac{\eta }{6c}\right) $.
\end{proof}

\section{Homogenization of problem (\protect\ref{eq1.2a})-(\protect\ref%
{eq1.2d})}

\subsection{Preliminaries}

Let us notice that the $a_{ij}$ can be consider as functions defined on $%
\mathbb{R}_{y}^{N}\times \mathbb{R}_{\tau }\times \mathbb{R}_{z}^{N}$ that
thus not depend on the variable $\tau $.

We set%
\begin{equation*}
\begin{array}{c}
\mathbb{E}_{0}^{1}=L^{2}\left( 0,T;H_{0}^{1}\left( \Omega ;\mathbb{R}\right)
^{N}\right) \times L^{2}\left( Q;L_{per}^{2}\left( \mathcal{T}%
;H_{\#}^{1}\left( Y;\mathbb{R}\right) ^{N}\right) \right) \\ 
\times L^{2}\left( Q;L_{per}^{2}\left( \mathcal{T};L_{per}^{2}\left(
Y;H_{\#}^{1}\left( Z;\mathbb{R}\right) ^{N}\right) \right) \right) \text{.}%
\end{array}%
\end{equation*}%
$\mathbb{E}_{0}^{1}$ is the space of vector functions $\mathbf{v}=\left( 
\mathbf{v}_{0},\mathbf{v}_{1},\mathbf{v}_{2}\right) $ with $\mathbf{v}%
_{0}=\left( v_{0}^{j}\right) \in L^{2}\left( 0,T;H_{0}^{1}\left( \Omega ;%
\mathbb{R}\right) ^{N}\right) $, $\mathbf{v}_{1}=\left( v_{1}^{j}\right) \in
L^{2}\left( Q;L_{per}^{2}\left( \mathcal{T};H_{\#}^{1}\left( Y;\mathbb{R}%
\right) ^{N}\right) \right) $ and

$\mathbf{v}_{2}=\left( v_{2}^{j}\right) \in L^{2}\left( Q;L_{per}^{2}\left( 
\mathcal{T};L_{per}^{2}\left( Y;H_{\#}^{1}\left( Z;\mathbb{R}\right)
^{N}\right) \right) \right) $. Provided with the norm 
\begin{eqnarray*}
\left\Vert \mathbf{v}\right\Vert _{\mathbb{E}_{0}^{1}} &=&\left(
\sum_{i,j=1}^{N}\left( \left\Vert \frac{\partial v_{0}^{j}}{\partial x_{i}}%
\right\Vert _{L^{2}\left( Q\right) }^{2}+\left\Vert \frac{\partial v_{1}^{j}%
}{\partial y_{i}}\right\Vert _{L^{2}\left( Q;L_{per}^{2}\left( Y\times 
\mathcal{T}\right) \right) }^{2}+\left\Vert \frac{\partial v_{2}^{j}}{%
\partial z_{i}}\right\Vert _{L^{2}\left( Q;L_{per}^{2}\left( Y\times 
\mathcal{T\times }Z\right) \right) }^{2}\right) \right) ^{\frac{1}{2}} \\
&&\left( \mathbf{v=}\left( \mathbf{v}_{0},\mathbf{v}_{1},\mathbf{v}%
_{2}\right) \in \mathbb{E}_{0}^{1}\right) \text{,}
\end{eqnarray*}%
$\mathbb{E}_{0}^{1}$ is a Hilbert space. Moreover, we have the following
lemma:

\begin{lemma}
\label{lem3.1} The vector space 
\begin{eqnarray*}
E_{0}^{\infty } &=&\mathcal{D}\left( Q;\mathbb{R}\right) ^{N}\times \left[ 
\mathcal{D}\left( Q;\mathbb{R}\right) \otimes \left( \mathcal{C}%
_{per}^{\infty }\left( \mathcal{T};\mathbb{R}\right) \otimes \left( \mathcal{%
C}_{per}^{\infty }\left( Y;\mathbb{R}\right) /\mathbb{C}\right) ^{N}\right) %
\right] \\
&&\times \left[ \mathcal{D}\left( Q;\mathbb{R}\right) \otimes \left[ 
\mathcal{C}_{per}^{\infty }\left( \mathcal{T};\mathbb{R}\right) \otimes
\left( \mathcal{C}_{per}^{\infty }\left( Y;\mathbb{R}\right) \otimes \left( 
\mathcal{C}_{per}^{\infty }\left( Z;\mathbb{R}\right) \mathfrak{/}\mathbb{C}%
\right) ^{N}\right) \right] \right]
\end{eqnarray*}%
(where $\mathcal{C}_{per}^{\infty }\left( \mathcal{T};\mathbb{R}\right) =%
\mathcal{C}_{per}\left( \mathcal{T};\mathbb{R}\right) \cap \mathcal{C}%
^{\infty }\left( \mathbb{R}\right) $,$\ \mathcal{C}_{per}^{\infty }\left( Y;%
\mathbb{R}\right) =\mathcal{C}_{per}\left( Y;\mathbb{R}\right) \cap \mathcal{%
C}^{\infty }\left( \mathbb{R}^{N}\right) $, $\mathcal{C}_{per}^{\infty
}\left( Y;\mathbb{R}\right) /\mathbb{C=}\left\{ v\in \mathcal{C}%
_{per}^{\infty }\left( Y;\mathbb{R}\right) :\int_{Y}v\left( y\right)
dy=0\right\} $, $\mathcal{C}_{per}^{\infty }\left( Z;\mathbb{R}\right) =%
\mathcal{C}_{per}\left( Z;\mathbb{R}\right) \cap \mathcal{C}^{\infty }\left( 
\mathbb{R}^{N}\right) $ and $\mathcal{C}_{per}^{\infty }\left( Z;\mathbb{R}%
\right) /\mathbb{C=}\left\{ v\in \mathcal{C}_{per}^{\infty }\left( Z;\mathbb{%
R}\right) :\int_{Z}v\left( y\right) dy=0\right\} $) is a dense subspace of $%
\mathbb{E}_{0}^{1}$.
\end{lemma}

We consider also the vector space 
\begin{eqnarray*}
\mathbb{E}_{0}^{1,\tau } &=&H_{0}^{1}\left( \Omega ;\mathbb{R}\right)
^{N}\times L^{2}\left( \Omega ;L_{per}^{2}\left( \mathcal{T}%
;H_{\#}^{1}\left( Y;\mathbb{R}\right) ^{N}\right) \right) \\
&&\times L^{2}\left( \Omega ;L_{per}^{2}\left( \mathcal{T};L_{per}^{2}\left(
Y;H_{\#}^{1}\left( Z;\mathbb{R}\right) ^{N}\right) \right) \right) \text{,}
\end{eqnarray*}%
provided with the norm 
\begin{eqnarray*}
\left\Vert \mathbf{v}\right\Vert _{\mathbb{E}_{0}^{1,\tau }} &=&\left(
\sum_{i,j=1}^{N}\left( \left\Vert \frac{\partial v_{0}^{j}}{\partial x_{i}}%
\right\Vert _{L^{2}\left( \Omega \right) }^{2}+\left\Vert \frac{\partial
v_{1}^{j}}{\partial y_{i}}\right\Vert _{L^{2}\left( \Omega
;L_{per}^{2}\left( Y\times \mathcal{T}\right) \right) }^{2}+\left\Vert \frac{%
\partial v_{2}^{j}}{\partial z_{i}}\right\Vert _{L^{2}\left( \Omega
;L_{per}^{2}\left( Y\times \mathcal{T\times }Z\right) \right) }^{2}\right)
\right) ^{\frac{1}{2}} \\
&&\left( \mathbf{v=}\left( \mathbf{v}_{0},\mathbf{v}_{1},\mathbf{v}%
_{2}\right) \in \mathbb{E}_{0}^{1,\tau }\right)
\end{eqnarray*}%
which makes it a Hilbert space. Furthermore,%
\begin{equation*}
\mathbb{E}_{0}^{1}=L^{2}\left( 0,T;\mathbb{E}_{0}^{1,\tau }\right) \text{.}
\end{equation*}

For any $1\leq k\leq N$, let us put 
\begin{equation*}
\mathbf{v}^{k}\mathbf{=}\left( v_{0}^{k},v_{1}^{k},v_{2}^{k}\right) \text{,
where }\mathbf{v=}\left( \mathbf{v}_{0},\mathbf{v}_{1},\mathbf{v}_{2}\right)
\in \mathbb{E}_{0}^{1,\tau }\text{ with }\mathbf{v}_{0}=\left(
v_{0}^{j}\right) \text{, }\mathbf{v}_{1}=\left( v_{1}^{j}\right) \text{ and }%
\mathbf{v}_{2}=\left( v_{2}^{j}\right) \text{,}
\end{equation*}%
and%
\begin{equation*}
\mathbb{D}_{i}\mathbf{v}^{k}=\frac{\partial v_{0}^{k}}{\partial x_{i}}+\frac{%
\partial v_{1}^{k}}{\partial y_{i}}+\frac{\partial v_{2}^{k}}{\partial z_{i}}%
\text{\qquad }\left( 1\leq i\leq N\right) \text{.}
\end{equation*}%
We are led to 
\begin{equation*}
\left\Vert \mathbf{v}\right\Vert _{\mathbb{E}_{0}^{1,\tau }}=\left[
\sum_{j,k=1}^{N}\left\Vert \mathbb{D}_{j}\mathbf{v}^{k}\right\Vert
_{L^{2}\left( \Omega ;L_{per}^{2}\left( Y\times \mathcal{T\times }Z\right)
\right) }^{2}\right] ^{\frac{1}{2}}\text{\qquad }\left( \mathbf{v=}\left( 
\mathbf{v}_{0},\mathbf{v}_{1},\mathbf{v}_{2}\right) \in \mathbb{E}%
_{0}^{1,\tau }\right) \text{.}
\end{equation*}%
On $\mathbb{E}_{0}^{1,\tau }\times \mathbb{E}_{0}^{1,\tau }$ we consider de
bilinear form $a_{\Omega }\left( .,.\right) $ defined by 
\begin{equation*}
a_{\Omega }\left( \mathbf{u},\mathbf{v}\right) =\sum_{i,j,k=1}^{N}\int \int
\int \int_{\Omega \times Y\times \mathcal{T\times }Z}a_{ij}\mathbb{D}_{i}%
\mathbf{u}^{k}\mathbb{D}_{i}\mathbf{v}^{k}dxdydzd\tau \text{ for all }\left( 
\mathbf{u},\mathbf{v}\right) \in \mathbb{E}_{0}^{1,\tau }\times \mathbb{E}%
_{0}^{1,\tau }\text{.}
\end{equation*}

\begin{remark}
\label{rem3.1} In view (\ref{eq1.1a})-(\ref{eq1.1c}), $a_{\Omega }\left(
.,.\right) $\ is continuous symmetric and $E_{0}^{1,\tau }$-coercive.
\end{remark}

We recall that for $\Omega $ sufficiently smooth $V=\left\{ \mathbf{v}\in
H_{0}^{1}\left( \Omega ;\mathbb{R}\right) ^{N}:div\mathbf{v}=0\right\} $.
For fixed $0<\varepsilon <1$, we introduce the bilinear form $a^{\varepsilon
}$ on $H_{0}^{1}\left( \Omega ;\mathbb{R}\right) ^{N}\times H_{0}^{1}\left(
\Omega ;\mathbb{R}\right) ^{N}$ defined by%
\begin{equation*}
a^{\varepsilon }\left( \mathbf{u},\mathbf{v}\right)
=\sum_{k=1}^{N}\sum_{i,j=1}^{N}\int_{\Omega }a_{ij}^{\varepsilon }\frac{%
\partial u^{k}}{\partial x_{j}}\frac{\partial v^{k}}{\partial x_{i}}dx
\end{equation*}%
for $\mathbf{u}=\left( u^{k}\right) $ and $\mathbf{v}=\left( v^{k}\right) $
in $H_{0}^{1}\left( \Omega ;\mathbb{R}\right) ^{N}$. By virtue of (\ref%
{eq1.1a})-(\ref{eq1.1c}), the form $a^{\varepsilon }$ is symmetric,
continuous and $H_{0}^{1}\left( \Omega ;\mathbb{R}\right) ^{N}$-coercive. We
introduce also the trilinear form $b$ on $H_{0}^{1}\left( \Omega ;\mathbb{R}%
\right) ^{N}\times H_{0}^{1}\left( \Omega ;\mathbb{R}\right) ^{N}\times
H_{0}^{1}\left( \Omega ;\mathbb{R}\right) ^{N}$ defined by%
\begin{equation*}
b\left( \mathbf{u},\mathbf{v},\mathbf{w}\right)
=\sum_{k=1}^{N}\sum_{j=1}^{N}\int_{\Omega }u^{j}\frac{\partial v^{k}}{%
\partial x_{j}}w^{k}dx
\end{equation*}%
for $\mathbf{u}=\left( u^{k}\right) $, $\mathbf{v}=\left( v^{k}\right) $ and 
$\mathbf{w}=\left( w^{k}\right) \in H_{0}^{1}\left( \Omega ;\mathbb{R}%
\right) ^{N}$. The form $b$ has the following nice properties \cite[%
pp.162-163]{bib14}:%
\begin{equation}
\left\vert b\left( \mathbf{u},\mathbf{v},\mathbf{w}\right) \right\vert \leq
c\left( N\right) \left\Vert \mathbf{u}\right\Vert _{H_{0}^{1}\left( \Omega
\right) ^{N}}\left\Vert \mathbf{v}\right\Vert _{H_{0}^{1}\left( \Omega
\right) ^{N}}\left\Vert \mathbf{w}\right\Vert _{H_{0}^{1}\left( \Omega
\right) ^{N}}  \label{eq2.5a}
\end{equation}%
for all $\mathbf{u}$, $\mathbf{v}$ and $\mathbf{w}\in H_{0}^{1}\left( \Omega
;\mathbb{R}\right) ^{N}$, where the positive constant $c\left( N\right) $
depends on $N$ and $\Omega $, and where 
\begin{equation*}
\left\Vert \mathbf{v}\right\Vert _{H_{0}^{1}\left( \Omega \right)
^{N}}=\left( \sum_{k=1}^{N}\int_{\Omega }\left\vert \nabla v^{k}\right\vert
dx\right) ^{\frac{1}{2}}
\end{equation*}%
with $\nabla v^{k}=\left( \frac{\partial v^{k}}{\partial x_{1}},...,\frac{%
\partial v^{k}}{\partial x_{N}}\right) $;%
\begin{equation}
b\left( \mathbf{u},\mathbf{v},\mathbf{v}\right) =0\qquad \left( \mathbf{u}%
\in V\text{, }\mathbf{v}\in H_{0}^{1}\left( \Omega ;\mathbb{R}\right)
^{N}\right)  \label{eaq2.5b}
\end{equation}%
and%
\begin{equation}
b\left( \mathbf{u},\mathbf{v},\mathbf{w}\right) =-b\left( \mathbf{u},\mathbf{%
w},\mathbf{v}\right) \qquad \left( \mathbf{u}\in V\text{, }\mathbf{v}\text{
and }\mathbf{w}\in H_{0}^{1}\left( \Omega ;\mathbb{R}\right) ^{N}\right) 
\text{.}  \label{eq2.5c}
\end{equation}

\noindent Let us set 
\begin{equation*}
\mathbb{E}_{0}^{\tau }=V\times L^{2}\left( \Omega ;L_{per}^{2}\left( 
\mathcal{T};W_{y}\right) \right) \times L^{2}\left( \Omega
;L_{per}^{2}\left( \mathcal{T};L_{per}^{2}\left( Y;W_{z}\right) \right)
\right) \text{, }
\end{equation*}%
where 
\begin{equation*}
W_{y}=\left\{ \mathbf{u}\in H_{\#}^{1}\left( Y;\mathbb{R}\right) ^{N}:div_{y}%
\mathbf{u}=0\right\} \text{,}
\end{equation*}%
\begin{equation*}
W_{z}=\left\{ \mathbf{u}\in H_{\#}^{1}\left( Z;\mathbb{R}\right) ^{N}:div_{z}%
\mathbf{u}=0\right\} \text{,}
\end{equation*}%
and where $div_{y}\mathbf{u}=\sum_{j=1}^{N}\frac{\partial u^{j}}{\partial
y_{j}}$ for $\mathbf{u}=\left( u^{j}\right) \in H_{\#}^{1}\left( Y\right)
^{N}$ and $div_{z}\mathbf{u}=\sum_{j=1}^{N}\frac{\partial u^{j}}{\partial
z_{j}}$ for $\mathbf{u}=\left( u^{j}\right) \in H_{\#}^{1}\left( Z\right)
^{N}$. $\mathbb{E}_{0}^{\tau }$ is a closed vector subspace of $\mathbb{E}%
_{0}^{1,\tau }$.

\begin{lemma}
\label{lem3.2} Let $\mathbf{u=}\left( \mathbf{u}_{0},\mathbf{u}_{1},\mathbf{u%
}_{2}\right) \in \mathbb{E}_{0}^{1}$ such that $\mathbf{u}_{0}\in \mathcal{W}%
\left( 0,T\right) $. Suppose $\mathbf{u}$ is a solution to the variational
problem%
\begin{equation}
\mathbf{u}\in L^{2}\left( 0,T;\mathbb{E}_{0}^{\tau }\right) \text{,}
\label{eq3.1a}
\end{equation}%
\begin{equation}
\mathbf{u}_{0}\left( 0\right) =0\text{,}  \label{eq3.1b}
\end{equation}%
\begin{equation}
\left\{ 
\begin{array}{c}
\int_{0}^{T}\left\langle \mathbf{u}_{0}^{\prime }\left( t\right) ,\mathbf{v}%
_{0}\left( t\right) \right\rangle dt+\int_{0}^{T}a_{\Omega }\left( \mathbf{u}%
\left( t\right) ,\mathbf{v}\left( t\right) \right) dt+\int_{0}^{T}b\left( 
\mathbf{u}_{0}\left( t\right) ,\mathbf{u}_{0}\left( t\right) ,\mathbf{v}%
_{0}\left( t\right) \right) dt \\ 
=\int_{0}^{T}\left( \mathbf{f}\left( t\right) ,\mathbf{v}_{0}\left( t\right)
\right) dt \\ 
\text{for all }\mathbf{v=}\left( \mathbf{v}_{0},\mathbf{v}_{1},\mathbf{v}%
_{2}\right) \in L^{2}\left( 0,T;\mathbb{E}_{0}^{\tau }\right) \text{.}\qquad
\qquad \qquad \qquad \qquad \qquad%
\end{array}%
\right.  \label{eq3.1c}
\end{equation}%
Then $\mathbf{u}$ is unique.
\end{lemma}

\begin{proof}
Taking in account Remark \ref{rem3.1}, the proof of this lemma carries over
mutatis mutandis the one in \cite[Lemma 2.3]{bib11}.
\end{proof}

In the following lines, most of the notations are of \cite{bib11} and \cite%
{bib12}. Let us notice that $E_{0}^{\infty }$ is the space of functions of
the form 
\begin{equation}
\mathbf{\Phi =}\left( \mathbf{\psi }_{0},\mathbf{\psi }_{1},\mathbf{\psi }%
_{2}\right)  \label{eq3.2a}
\end{equation}%
with $\mathbf{\psi }_{0}=\left( \psi _{0}^{j}\right) \in \mathcal{D}\left( Q;%
\mathbb{R}\right) ^{N}$, $\mathbf{\psi }_{1}=\left( \psi _{1}^{j}\right) \in 
\mathcal{D}\left( Q;\mathbb{R}\right) \otimes \left( \mathcal{C}%
_{per}^{\infty }\left( \mathcal{T};\mathbb{R}\right) \otimes \left( \mathcal{%
C}_{per}^{\infty }\left( Y;\mathbb{R}\right) /\mathbb{C}\right) ^{N}\right) $
and $\mathbf{\psi }_{2}=\left( \psi _{2}^{j}\right) \in \mathcal{D}\left( Q;%
\mathbb{R}\right) \otimes \left[ \mathcal{C}_{per}^{\infty }\left( \mathcal{T%
};\mathbb{R}\right) \otimes \left( \mathcal{C}_{per}^{\infty }\left( Y;%
\mathbb{R}\right) \otimes \left( \mathcal{C}_{per}^{\infty }\left( Z;\mathbb{%
R}\right) \mathfrak{/}\mathbb{C}\right) ^{N}\right) \right] $. Let $\mathbf{%
\Phi }$ be as in (\ref{eq3.2a}). For $\varepsilon >0$ we set 
\begin{equation}
\mathbf{\Phi }_{\varepsilon }=\mathbf{\psi }_{0}+\varepsilon \mathbf{\psi }%
_{1}^{\varepsilon }+\varepsilon ^{2}\mathbf{\psi }_{2}^{\varepsilon }\text{,}
\label{eq3.2b}
\end{equation}%
i.e., $\Phi _{\varepsilon }^{j}\left( x,t\right) =\psi _{0}^{j}\left(
x,t\right) +\varepsilon \psi _{1}^{j}\left( x,t,\frac{x}{\varepsilon },\frac{%
t}{\varepsilon }\right) +\varepsilon ^{2}\psi _{2}^{j}\left( x,t,\frac{x}{%
\varepsilon },\frac{t}{\varepsilon },\frac{x}{\varepsilon ^{2}}\right) \quad
\left( \left( x,t\right) \in Q\text{, }1\leq j\leq N\right) $.

\begin{lemma}
\label{lem3.3} Let $\left( \mathbf{u}_{\varepsilon }\right) _{\varepsilon
\in E}=\left( u_{\varepsilon }^{1},...,u_{\varepsilon }^{N}\right)
_{\varepsilon \in E}\subset L^{2}\left( 0,T;H_{0}^{1}\left( \Omega ;\mathbb{R%
}\right) \right) ^{N}$, where $E$ is fundamental sequence. Suppose that as $%
E\ni \varepsilon \rightarrow 0$%
\begin{equation}
\frac{\partial u_{\varepsilon }^{k}}{\partial x_{j}}\rightarrow \mathbb{D}%
_{j}\mathbf{u}^{k}\text{ reiteratively in }L^{2}\left( Q\right) \text{ }%
\Sigma \text{-weak }\left( 1\leq j,k\leq N\right) \text{,}  \label{eq3.3a}
\end{equation}%
where $\mathbf{u=}\left( \mathbf{u}_{0},\mathbf{u}_{1},\mathbf{u}_{2}\right)
\in \mathbb{E}_{0}^{1}$. Then, 
\begin{equation}
\int_{0}^{T}a^{\varepsilon }\left( \mathbf{u}_{\varepsilon }\left( t\right) ,%
\mathbf{\Phi }_{\varepsilon }\left( t\right) \right) dt\rightarrow
\int_{0}^{T}a_{\Omega }\left( \mathbf{u}\left( t\right) ,\mathbf{\Phi }%
\left( t\right) \right) dt  \label{eq3.3b}
\end{equation}%
as $E\ni \varepsilon \rightarrow 0$. Moreover, if $\mathbf{u}_{\varepsilon
}\rightarrow \mathbf{u}_{0}$ in $L^{2}\left( Q\right) ^{N}$ as $E\ni
\varepsilon \rightarrow 0$ then%
\begin{equation}
\int_{0}^{T}b\left( \mathbf{u}_{\varepsilon }\left( t\right) ,\mathbf{u}%
_{\varepsilon }\left( t\right) ,\mathbf{\Phi }_{\varepsilon }\left( t\right)
\right) dt\rightarrow \int_{0}^{T}b\left( \mathbf{u}_{0}\left( t\right) ,%
\mathbf{u}_{0}\left( t\right) ,\mathbf{\psi }_{0}\left( t\right) \right) dt
\label{eq3.3c}
\end{equation}%
as $E\ni \varepsilon \rightarrow 0$.
\end{lemma}

\begin{proof}[Preuve]
Let us first note that 
\begin{equation*}
\frac{\partial \Phi _{\varepsilon }^{k}}{\partial x_{i}}=\frac{\partial \psi
_{0}^{k}}{\partial x_{i}}+\varepsilon \left( \frac{\partial \psi _{1}^{k}}{%
\partial x_{i}}\right) ^{\varepsilon }+\left( \frac{\partial \psi _{1}^{k}}{%
\partial y_{i}}\right) ^{\varepsilon }+\varepsilon ^{2}\left( \frac{\partial
\psi _{2}^{k}}{\partial x_{i}}\right) ^{\varepsilon }+\varepsilon \left( 
\frac{\partial \psi _{2}^{k}}{\partial y_{i}}\right) ^{\varepsilon }+\left( 
\frac{\partial \psi _{2}^{k}}{\partial z_{i}}\right) ^{\varepsilon }
\end{equation*}%
for $\varepsilon >0$ and $1\leq i,k\leq N$. Thus, for any real number $p\geq
1$, 
\begin{equation}
\frac{\partial \Phi _{\varepsilon }^{k}}{\partial x_{i}}\rightarrow \mathbb{D%
}_{i}\mathbf{\Phi }^{k}=\frac{\partial \psi _{0}^{k}}{\partial x_{i}}+\frac{%
\partial \psi _{1}^{k}}{\partial y_{i}}+\frac{\partial \psi _{2}^{k}}{%
\partial z_{i}}\text{ reiteratively in }L^{p}\left( Q\right) \text{ }\Sigma 
\text{-strong}  \label{eq3.4}
\end{equation}%
as $\varepsilon \rightarrow 0$ (Example \ref{ex2.1b} or Remark \ref{rem2.1}%
). Let 
\begin{equation*}
z_{\varepsilon }^{ijk}=\frac{\partial u_{\varepsilon }^{k}}{\partial x_{j}}%
\frac{\partial \Phi _{\varepsilon }^{k}}{\partial x_{i}}\text{ }\left(
\varepsilon \in E\text{ and }1\leq i,j,k\leq N\right) \text{.}
\end{equation*}%
For any $p\geq 2$, $z_{\varepsilon }^{ijk}\in L^{r}\left( Q\right) $ with $%
\frac{1}{r}=\frac{1}{2}+\frac{1}{p}$, and in virtue of Proposition \ref%
{pr2.6} it follows from (\ref{eq3.3a}) and (\ref{eq3.4}) that 
\begin{equation*}
z_{\varepsilon }^{ijk}\rightarrow \mathbb{D}_{j}\mathbf{u}^{k}\mathbb{D}_{i}%
\mathbf{\Phi }^{k}\text{ reiteratively in }L^{r}\left( Q\right) \text{ }%
\Sigma \text{-weak as }E\ni \varepsilon \rightarrow 0\text{.}
\end{equation*}%
Particularly, 
\begin{equation}
\int_{Q}z_{\varepsilon }^{ijk}f^{\varepsilon }dxdt\rightarrow \int \int \int
\int_{Q\times Y\times \mathcal{T}\times Z}\mathbb{D}_{j}\mathbf{u}^{k}%
\mathbb{D}_{i}\mathbf{\Phi }^{k}fdxdtdydzd\tau  \label{eq3.5}
\end{equation}%
for any $f\in \mathcal{K}\left( Q;\mathcal{C}_{per}\left( Y\times \mathcal{T}%
\times Z\right) \right) $ as $E\ni \varepsilon \rightarrow 0$. Further, $%
\left( \frac{\partial u_{\varepsilon }^{k}}{\partial x_{j}}\right)
_{\varepsilon \in E}$ is bounded in $L^{2}\left( Q\right) $, thus, $\left(
z_{\varepsilon }^{ijk}\right) _{\varepsilon \in E}$ is bounded too in $%
L^{2}\left( Q\right) $. Using Theorem \ref{th2.1}, a subsequence $E^{\prime
} $ can be extracted from $E$ such that in particular%
\begin{equation*}
\int_{Q}z_{\varepsilon }^{ijk}f^{\varepsilon }dxdt\rightarrow \int \int \int
\int_{Q\times Y\times \mathcal{T}\times Z}w_{0}fdxdtdydzd\tau \text{ as }%
E^{\prime }\ni \varepsilon \rightarrow 0
\end{equation*}%
for all $f\in \mathcal{K}\left( Q;\mathcal{C}_{per}\left( Y\times \mathcal{T}%
\times Z\right) \right) $, where $w_{0}\in L^{2}\left( Q;L_{per}^{2}\left(
Y\times \mathcal{T}\times Z\right) \right) $. Using (\ref{eq3.5}) we deduce
that $w_{0}=\mathbb{D}_{j}\mathbf{u}^{k}\mathbb{D}_{i}\mathbf{\Phi }^{k}$.
Thus, one can replace $f$ \ by $a_{ij}$ in (\ref{eq3.5}). Hence, (\ref%
{eq3.3b}) follows. Let us prove (\ref{eq3.3c}). Suppose $\mathbf{u}%
_{\varepsilon }\rightarrow \mathbf{u}_{0}$ in $L^{2}\left( Q\right) ^{N}$ as 
$E\ni \varepsilon \rightarrow 0$. Then by part (iii) of Example \ref{ex2.1b}
and Proposition \ref{pr2.6}, we have as $E\ni \varepsilon \rightarrow 0$%
\begin{equation*}
u_{\varepsilon }^{j}\frac{\partial u_{\varepsilon }^{k}}{\partial x_{j}}%
\rightarrow u_{0}^{j}\mathbb{D}_{j}\mathbf{u}^{k}\text{ reiteratively in }%
L^{1}\left( Q\right) \text{ }\Sigma \text{-weak }\left( 1\leq j,k\leq
N\right) \text{.}
\end{equation*}%
Thus, 
\begin{equation*}
\sum_{j,k=1}^{N}\int_{Q}u_{\varepsilon }^{j}\frac{\partial u_{\varepsilon
}^{k}}{\partial x_{j}}\psi _{0}^{k}dxdt\rightarrow
\sum_{j,k=1}^{N}\int_{Q}u_{0}^{j}\frac{\partial u_{0}^{k}}{\partial x_{j}}%
\psi _{0}^{k}dxdt\text{,}
\end{equation*}%
\begin{equation*}
\varepsilon \sum_{j,k=1}^{N}\int_{Q}u_{\varepsilon }^{j}\frac{\partial
u_{\varepsilon }^{k}}{\partial x_{j}}\left( \psi _{1}^{k}\right)
^{\varepsilon }dxdt\rightarrow 0\text{ and }\varepsilon
^{2}\sum_{j,k=1}^{N}\int_{Q}u_{\varepsilon }^{j}\frac{\partial
u_{\varepsilon }^{k}}{\partial x_{j}}\left( \psi _{2}^{k}\right)
^{\varepsilon }dxdt\rightarrow 0
\end{equation*}%
as $E\ni \varepsilon \rightarrow 0$. Hence, (\ref{eq3.3c}) follows.
\end{proof}

\subsection{A convergence result for (\protect\ref{eq1.2a})-(\protect\ref%
{eq1.2d})}

\begin{theorem}
\label{th3.1} Suppose $\mathbf{f}$, $\mathbf{f}^{\prime }\in L^{2}\left(
0,T;V^{\prime }\right) $ and $\mathbf{f}\left( 0\right) \in H$. For any $%
0<\varepsilon <1$, let $\mathbf{u}_{\varepsilon }=\left( u_{\varepsilon
}^{k}\right) $ be defined by (\ref{eq1.2a})-(\ref{eq1.2d}). Then as $%
\varepsilon \rightarrow 0$, 
\begin{equation}
\mathbf{u}_{\varepsilon }\rightarrow \mathbf{u}_{0}\text{ in }\mathcal{W}%
\left( 0,T\right) \text{-weak \qquad \qquad \qquad \qquad \qquad }
\label{eq3.6a}
\end{equation}%
and%
\begin{equation}
\frac{\partial u_{\varepsilon }^{k}}{\partial x_{j}}\rightarrow \mathbb{D}%
_{j}\mathbf{u}^{k}\text{ reiteratively in }L^{2}\left( Q\right) \text{ }%
\Sigma \text{-weak }  \label{eq3.6b}
\end{equation}%
$\left( 1\leq j,k\leq N\right) $, where $\mathbf{u=}\left( \mathbf{u}_{0},%
\mathbf{u}_{1},\mathbf{u}_{2}\right) $ (with $\mathbf{u}_{0}=\left(
u_{0}^{k}\right) $, $\mathbf{u}_{1}=\left( u_{1}^{k}\right) $ and $\mathbf{u}%
_{2}=\left( u_{2}^{k}\right) $) is the unique solution to the variational
problem (\ref{eq3.1a})-(\ref{eq3.1c}).
\end{theorem}

\begin{proof}[Preuve]
Let $E$ be a fundamental sequence. By \cite[Proposition 2.1]{bib11}, the
sequences $\left( \mathbf{u}_{\varepsilon }\right) _{\varepsilon \in E}$ and 
$\left( p_{\varepsilon }\right) _{\varepsilon \in E}$ are bounded in $%
\mathcal{W}\left( 0,T\right) $ and $L^{2}\left( 0,T;L^{2}\left( \Omega ;%
\mathbb{R}\right) \right) $, respectively. Thus, in virtue of Theorem \ref%
{th2.1} and Theorem \ref{th2.2}, there exists a subsequence $E^{\prime }$
extracted from $E$ and functions $\mathbf{u}_{0}=\left( u_{0}^{k}\right)
_{1\leq k\leq N}\in \mathcal{W}\left( 0,T\right) $, $\mathbf{u}_{1}=\left(
u_{1}^{k}\right) _{1\leq k\leq N}\in L^{2}\left( Q;L_{per}^{2}\left( 
\mathcal{T};H_{\#}^{1}\left( Y;\mathbb{R}\right) \right) \right) ^{N}$,

\noindent $\mathbf{u}_{2}=\left( u_{2}^{k}\right) _{1\leq k\leq N}\in
L^{2}\left( Q;L_{per}^{2}\left( \mathcal{T};L_{per}^{2}\left(
Y;H_{\#}^{1}\left( Z\right) \right) \right) \right) ^{N}$, $p\in L^{2}\left(
Q;L_{per}^{2}\left( Y\times \mathcal{T}\times Z\right) \right) $ such that
as $E^{\prime }\ni \varepsilon \rightarrow 0$, we have (\ref{eq3.6a})-(\ref%
{eq3.6b}) and 
\begin{equation}
p_{\varepsilon }\rightarrow p\text{ reiteratively in }L^{2}\left( Q\right) 
\text{ }\Sigma \text{-weak.}  \label{eq3.6c}
\end{equation}%
As the space $\mathcal{W}\left( 0,T\right) $ is compactly embedded in $%
L^{2}\left( Q\right) ^{N}$, one can extract the subsequence $E^{\prime }$
such that as $E^{\prime }\ni \varepsilon \rightarrow 0$, 
\begin{equation}
\mathbf{u}_{\varepsilon }\rightarrow \mathbf{u}_{0}\text{ in }L^{2}\left(
Q\right) ^{N}\text{. }  \label{eq3.6d}
\end{equation}%
To end this proof, it remains to show that $\mathbf{u=}\left( \mathbf{u}_{0},%
\mathbf{u}_{1},\mathbf{u}_{2}\right) $ verifies (\ref{eq3.1a})-(\ref{eq3.1c}%
). Hence, one could conclude that (\ref{eq3.6a})-(\ref{eq3.6b} hold as $%
\varepsilon \rightarrow 0$ in virtue of Lemma \ref{lem3.2}. For the proof of
(\ref{eq3.1a}), let us observe that $\func{div}_{x}\mathbf{u}_{\varepsilon
}\rightarrow \func{div}_{x}\mathbf{u}_{0}+\func{div}_{y}\mathbf{u}_{1}+\func{%
div}_{z}\mathbf{u}_{2}$ reiteratively in $L^{2}\left( Q\right) $ $\Sigma $%
-weak as $E^{\prime }\ni \varepsilon \rightarrow 0$ and $\func{div}_{x}%
\mathbf{u}_{\varepsilon }=\func{div}_{x}\mathbf{u}_{0}=0$. Moreover, by part
(i) of Proposition \ref{pr2.4}, $\func{div}_{x}\mathbf{u}_{\varepsilon }$
weakly two-scale converges to $\func{div}_{x}\mathbf{u}_{0}+\func{div}_{y}%
\mathbf{u}_{1}$ in $L^{2}\left( Q\right) $, as $E^{\prime }\ni \varepsilon
\rightarrow 0$. Thus, $\func{div}_{x}\mathbf{u}_{0}=\func{div}_{y}\mathbf{u}%
_{1}=\func{div}_{z}\mathbf{u}_{2}=0$, and (\ref{eq3.1a}) follows. One can
easily show as in \cite[Proof of Theorem 2.4]{bib11} that (\ref{eq3.1b})$%
\mathbf{\ }$holds. Let us check that $\mathbf{u=}\left( \mathbf{u}_{0},%
\mathbf{u}_{1},\mathbf{u}_{2}\right) $ verifies (\ref{eq3.1c}). For $%
\varepsilon >0$, let $\mathbf{\Phi }_{\varepsilon }$ be defined by (\ref%
{eq3.2b}) with (\ref{eq3.2a}). We have:%
\begin{equation}
\frac{\partial \Phi _{\varepsilon }^{k}}{\partial t}=\frac{\partial \psi
_{0}^{k}}{\partial t}+\varepsilon \left( \frac{\partial \psi _{1}^{k}}{%
\partial t}\right) ^{\varepsilon }+\left( \frac{\partial \psi _{1}^{k}}{%
\partial \tau }\right) ^{\varepsilon }+\varepsilon ^{2}\left( \frac{\partial
\psi _{2}^{k}}{\partial t}\right) ^{\varepsilon }+\varepsilon \left( \frac{%
\partial \psi _{2}^{k}}{\partial \tau }\right) ^{\varepsilon }\text{,}
\label{eq3.7a}
\end{equation}%
\begin{equation}
\frac{\partial \Phi _{\varepsilon }^{k}}{\partial x_{i}}=\frac{\partial \psi
_{0}^{k}}{\partial x_{i}}+\varepsilon \left( \frac{\partial \psi _{1}^{k}}{%
\partial x_{i}}\right) ^{\varepsilon }+\left( \frac{\partial \psi _{1}^{k}}{%
\partial y_{i}}\right) ^{\varepsilon }+\varepsilon ^{2}\left( \frac{\partial
\psi _{2}^{k}}{\partial x_{i}}\right) ^{\varepsilon }+\varepsilon \left( 
\frac{\partial \psi _{2}^{k}}{\partial y_{i}}\right) ^{\varepsilon }+\left( 
\frac{\partial \psi _{2}^{k}}{\partial z_{i}}\right) ^{\varepsilon }
\label{eq3.7b}
\end{equation}%
$\left( 1\leq i,k\leq N\right) $ and 
\begin{equation}
\func{div}_{x}\mathbf{\Phi }_{\varepsilon }=\func{div}_{x}\mathbf{\psi }%
_{0}+\varepsilon \left( \func{div}_{x}\mathbf{\psi }_{1}\right)
^{\varepsilon }+\left( \func{div}_{y}\mathbf{\psi }_{1}\right) ^{\varepsilon
}+\varepsilon ^{2}\left( \func{div}_{x}\mathbf{\psi }_{2}\right)
^{\varepsilon }+\varepsilon \left( \func{div}_{y}\mathbf{\psi }_{2}\right)
^{\varepsilon }+\left( \func{div}_{z}\mathbf{\psi }_{2}\right) ^{\varepsilon
}\text{.}  \label{eq3.7c}
\end{equation}%
Multiplying (\ref{eq1.2a}) by $\mathbf{\Phi }_{\varepsilon }$, one has%
\begin{equation}
\begin{array}{c}
\int_{0}^{T}\left\langle \mathbf{u}_{\varepsilon }^{\prime }\left( t\right) ,%
\mathbf{\Phi }_{\varepsilon }\left( t\right) \right\rangle
dt+\int_{0}^{T}a^{\varepsilon }\left( \mathbf{u}_{\varepsilon }\left(
t\right) ,\mathbf{\Phi }_{\varepsilon }\left( t\right) \right)
dt+\int_{0}^{T}b\left( \mathbf{u}_{\varepsilon }\left( t\right) ,\mathbf{u}%
_{\varepsilon }\left( t\right) ,\mathbf{\Phi }_{\varepsilon }\left( t\right)
\right) dt \\ 
-\int_{Q}p_{\varepsilon }div_{x}\mathbf{\Phi }_{\varepsilon
}dxdt=\int_{0}^{T}\left\langle \mathbf{f}\left( t\right) ,\mathbf{\Phi }%
_{\varepsilon }\left( t\right) \right\rangle dt\text{.}%
\end{array}
\label{eq3.8a}
\end{equation}%
Let us observe that by (\ref{eq3.4}) one easily has $\mathbf{\Phi }%
_{\varepsilon }\rightarrow \mathbf{\psi }_{0}$ in $L^{2}\left(
0,T;H_{0}^{1}\left( \Omega \right) ^{N}\right) $-weak as $\varepsilon
\rightarrow 0$. Thus, using (\ref{eq3.7a})-(\ref{eq3.7c}), a passage to the
limit in (\ref{eq3.8a}) as $E^{\prime }\ni \varepsilon \rightarrow 0$ leads
to 
\begin{equation*}
\begin{array}{c}
\int_{0}^{T}\left\langle \mathbf{u}_{0}^{\prime }\left( t\right) ,\mathbf{%
\psi }_{0}\left( t\right) \right\rangle dt+\int_{0}^{T}a_{\Omega }\left( 
\mathbf{u}\left( t\right) ,\mathbf{\Phi }\left( t\right) \right)
dt+\int_{0}^{T}b\left( \mathbf{u}_{0}\left( t\right) ,\mathbf{u}_{0}\left(
t\right) ,\mathbf{\psi }_{0}\left( t\right) \right) dt \\ 
-\int \int \int \int_{Q}p\left( div_{x}\mathbf{\psi }_{0}+div_{y}\mathbf{%
\psi }_{1}+div_{z}\mathbf{\psi }_{2}\right) dxdtdydzd\tau \\ 
=\int_{0}^{T}\left( \mathbf{f}\left( t\right) ,\mathbf{\psi }_{0}\left(
t\right) \right) dt\text{,}%
\end{array}%
\end{equation*}%
having taken in account (\ref{eq3.6c})-(\ref{eq3.6d}) and Lemma \ref{lem3.3}%
. But, by continuity we have 
\begin{equation}
\begin{array}{c}
\int_{0}^{T}\left\langle \mathbf{u}_{0}^{\prime }\left( t\right) ,\mathbf{v}%
_{0}\left( t\right) \right\rangle dt+\int_{0}^{T}a_{\Omega }\left( \mathbf{u}%
\left( t\right) ,\mathbf{v}\left( t\right) \right) dt+\int_{0}^{T}b\left( 
\mathbf{u}_{0}\left( t\right) ,\mathbf{u}_{0}\left( t\right) ,\mathbf{v}%
_{0}\left( t\right) \right) dt \\ 
-\int \int_{Q}p\left( div_{x}\mathbf{v}_{0}+div_{y}\mathbf{v}_{1}+div_{z}%
\mathbf{v}_{2}\right) dxdtdydzd\tau \\ 
=\int_{0}^{T}\left( \mathbf{f}\left( t\right) ,\mathbf{v}_{0}\left( t\right)
\right) dt\text{,}%
\end{array}
\label{eq3.8b}
\end{equation}%
for all $\mathbf{v=}\left( \mathbf{v}_{0},\mathbf{v}_{1},\mathbf{v}%
_{2}\right) \in \mathbb{E}_{0}^{1}$, thanks to Lemma \ref{lem3.1}. Finally,
taking a particular $\mathbf{v=}\left( \mathbf{v}_{0},\mathbf{v}_{1},\mathbf{%
v}_{2}\right) \in L^{2}\left( 0,T;\mathbb{E}_{0}^{\tau }\right) $ in (\ref%
{eq3.8b}) leads to (\ref{eq3.1c}). The theorem follows.
\end{proof}

Now, let us put%
\begin{equation*}
E_{y}=L_{per}^{2}\left( \mathcal{T};H_{\#}^{1}\left( Y;\mathbb{R}\right)
^{N}\right) \text{,}
\end{equation*}%
\begin{equation*}
E_{z}=L_{per}^{2}\left( \mathcal{T};L_{per}^{2}\left( Y;H_{\#}^{1}\left( Z;%
\mathbb{R}\right) ^{N}\right) \right) \text{.}
\end{equation*}%
Provided with the norm 
\begin{equation*}
\left\Vert \left( \mathbf{v}_{1},\mathbf{v}_{2}\right) \right\Vert
_{E_{y}\times E_{z}}=\left[ \left\Vert \mathbf{v}_{1}\right\Vert
_{L_{per}^{2}\left( \mathcal{T};H_{\#}^{1}\left( Y;\mathbb{R}\right)
^{N}\right) }^{2}+\left\Vert \mathbf{v}_{2}\right\Vert _{L_{per}^{2}\left( 
\mathcal{T};L_{per}^{2}\left( Y;H_{\#}^{1}\left( Z;\mathbb{R}\right)
^{N}\right) \right) }^{2}\right] ^{\frac{1}{2}}
\end{equation*}

for $\left( \mathbf{v}_{1},\mathbf{v}_{2}\right) \in E_{y}\times E_{z}$, $%
E_{y}\times E_{z}$ is a Hilbert. For $\left( \mathbf{u}_{1},\mathbf{u}%
_{2}\right) $ and $\left( \mathbf{v}_{1},\mathbf{v}_{2}\right) \in
E_{y}\times E_{z}$ with $\mathbf{u}_{1}\mathbf{=}\left( u_{1}^{k}\right) $, $%
\mathbf{u}_{2}\mathbf{=}\left( v_{2}^{k}\right) $, $\mathbf{v}_{1}=\left(
v_{1}^{k}\right) $ and $\mathbf{v}_{2}=\left( v_{2}^{k}\right) $, we set%
\begin{equation*}
a\left( \left( \mathbf{u}_{1},\mathbf{u}_{2}\right) ,\left( \mathbf{v}_{1},%
\mathbf{v}_{2}\right) \right) =\sum_{i,j,k=1}^{N}\int \int \int_{Y\times 
\mathcal{T}\times Z}a_{ij}\left( \frac{\partial u_{1}^{k}}{\partial y_{j}}+%
\frac{\partial u_{2}^{k}}{\partial z_{j}}\right) \left( \frac{\partial
v_{1}^{k}}{\partial y_{i}}+\frac{\partial v_{2}^{k}}{\partial z_{i}}\right)
dydzd\tau \text{.}
\end{equation*}%
This defines a bilinear form on $\left( E_{y}\times E_{z}\right) \times
\left( E_{y}\times E_{z}\right) $ which is continuous symmetric and $%
E_{y}\times E_{z}$-coercive with 
\begin{equation}
a\left( \left( \mathbf{v}_{1},\mathbf{v}_{2}\right) ,\left( \mathbf{v}_{1},%
\mathbf{v}_{2}\right) \right) \geq \alpha \left\Vert \left( \mathbf{v}_{1},%
\mathbf{v}_{2}\right) \right\Vert _{E_{y}\times E_{z}}^{2}  \label{eq3.9a}
\end{equation}%
for all $\left( \mathbf{v}_{1},\mathbf{v}_{2}\right) \in E_{y}\times E_{z}$,
where $\alpha $ is the constant in (\ref{eq1.1c}).

For $1\leq i,k\leq N$, we consider the variational problem%
\begin{equation}
\left\{ 
\begin{array}{c}
\left( \mathbf{\chi }_{ik},\mathbf{\eta }_{ik}\right) \in L_{per}^{2}\left( 
\mathcal{T};W_{y}\right) \times L_{per}^{2}\left( \mathcal{T}%
;L_{per}^{2}\left( Y;W_{z}\right) \right) :\qquad \qquad \qquad \quad \qquad
\qquad \\ 
a\left( \left( \mathbf{\chi }_{ik},\mathbf{\eta }_{ik}\right) ,\left( 
\mathbf{v}_{1},\mathbf{v}_{2}\right) \right) =\sum_{l=1}^{N}\int \int
\int_{Y\times \mathcal{T}\times Z}a_{il}\left( \frac{\partial v_{1}^{k}}{%
\partial y_{l}}+\frac{\partial v_{2}^{k}}{\partial z_{l}}\right) dydzd\tau
\qquad \qquad \qquad \\ 
\text{for all }\left( \mathbf{v}_{1},\mathbf{v}_{2}\right) \in
L_{per}^{2}\left( \mathcal{T};W_{y}\right) \times L_{per}^{2}\left( \mathcal{%
T};L_{per}^{2}\left( Y;W_{z}\right) \right) \qquad \qquad \qquad \qquad \quad%
\end{array}%
\right.  \label{eq3.9b}
\end{equation}%
which defines in a unique manner $\left( \mathbf{\chi }_{ik},\mathbf{\eta }%
_{ik}\right) $ with $\mathbf{\chi }_{ik}=\left( \mathcal{\chi }%
_{ik}^{l}\right) $ and $\mathbf{\eta }_{ik}=\left( \mathcal{\eta }%
_{ik}^{l}\right) $.

\begin{proposition}
\label{pr3.1} Under the hypotheses of Theorem \ref{th3.1}, we have 
\begin{equation}
\mathbf{u}_{1}\left( x,t,y,\tau \right) =-\sum_{i,k=1}^{N}\frac{\partial
u_{0}^{k}}{\partial x_{i}}\left( x,t\right) \mathbf{\chi }_{ik}\left( y,\tau
\right)  \label{eq3.10a}
\end{equation}%
for almost all $\left( x,t\right) \in Q$ and for almost all $\left( y,\tau
\right) \in Y\times \mathcal{T}$;%
\begin{equation}
\mathbf{u}_{2}\left( x,t,y,\tau ,z\right) =-\sum_{i,k=1}^{N}\frac{\partial
u_{0}^{k}}{\partial x_{i}}\left( x,t\right) \mathbf{\eta }_{ik}\left( y,\tau
,z\right)  \label{eq3.10b}
\end{equation}%
for almost all $\left( x,t\right) \in Q$ and for almost all $\left( y,\tau
,z\right) \in Y\times \mathcal{T}\times Z$.
\end{proposition}

\begin{proof}
In (\ref{eq3.1c}), we take a particular test function $\mathbf{v=}\left( 
\mathbf{v}_{0},\mathbf{v}_{1},\mathbf{v}_{2}\right) $ with $\mathbf{v}_{0}=0$%
, $\mathbf{v}_{1}=\varphi \otimes \mathbf{w}_{1}$ and $\mathbf{v}%
_{2}=\varphi \otimes \mathbf{w}_{2}$, where $\varphi \in \mathcal{D}\left( Q;%
\mathbb{R}\right) $ and $\left( \mathbf{w}_{1},\mathbf{w}_{2}\right) \in
L_{per}^{2}\left( \mathcal{T};W_{y}\right) \times L_{per}^{2}\left( \mathcal{%
T};L_{per}^{2}\left( Y;W_{z}\right) \right) $. This leads to 
\begin{equation}
\left\{ 
\begin{array}{c}
a\left( \left( \mathbf{u}_{1}\left( x,t\right) ,\mathbf{u}_{2}\left(
x,t\right) \right) ,\left( \mathbf{w}_{1},\mathbf{w}_{2}\right) \right)
=-\sum_{i,j,k=1}^{N}\frac{\partial u_{0}^{k}}{\partial x_{j}}\left(
x,t\right) \int \int \int_{Y\times \mathcal{T}\times Z}a_{ij}\left( \frac{%
\partial w_{1}^{k}}{\partial y_{i}}+\frac{\partial w_{2}^{k}}{\partial z_{i}}%
\right) dydzd\tau \\ 
\text{for all }\left( \mathbf{w}_{1},\mathbf{w}_{2}\right) \in
L_{per}^{2}\left( \mathcal{T};W_{y}\right) \times L_{per}^{2}\left( \mathcal{%
T};L_{per}^{2}\left( Y;W_{z}\right) \right) \text{,}\qquad \qquad%
\end{array}%
\right.  \label{eq3.10c}
\end{equation}%
for almost all $\left( x,t\right) \in Q$. But, in virtue of (\ref{eq3.9a})
the couple $\left( \mathbf{u}_{1}\left( x,t\right) ,\mathbf{u}_{2}\left(
x,t\right) \right) $ (for fixed $\left( x,t\right) \in Q$) is the unique
function in $L_{per}^{2}\left( \mathcal{T};W_{y}\right) \times
L_{per}^{2}\left( \mathcal{T};L_{per}^{2}\left( Y;W_{z}\right) \right) $
satisfying (\ref{eq3.10c}). However, by (\ref{eq3.9b}) one easily observes
that $\left( \mathbf{z}_{1}\left( x,t\right) ,\mathbf{z}_{2}\left(
x,t\right) \right) $ with $\mathbf{z}_{1}\left( x,t\right) =-\sum_{i,k=1}^{N}%
\frac{\partial u_{0}^{k}}{\partial x_{i}}\left( x,t\right) \mathbf{\chi }%
_{ik}$ and $\mathbf{z}_{2}\left( x,t\right) =-\sum_{i,k=1}^{N}\frac{\partial
u_{0}^{k}}{\partial x_{i}}\left( x,t\right) \mathbf{\eta }_{ik}$, also
verifies the variational equation (\ref{eq3.10c}). Hence, (\ref{eq3.10a})-(%
\ref{eq3.10b}) follows.
\end{proof}

\subsection{The macroscopic homogenized equations}

Our goal here is to derive de macroscopic homogenized model verified by the
couple $\left( \mathbf{u}_{0},p_{0}\right) $, where $\mathbf{u}_{0}$ is the
limit in (\ref{eq3.6a}) and $p_{0}$ is the mean of $p$ (in (\ref{eq3.6c})),
i.e., $p_{0}\left( x,t\right) =\int \int \int_{Y\times \mathcal{T}\times
Z}p\left( x,t,y,\tau ,z\right) dydzd\tau $ for $\left( x,t\right) \in Q$.

For $1\leq i,j,k,h\leq N$, we set%
\begin{equation}
q_{ijkh}=\delta _{kh}\int \int_{Y\times Z}a_{ij}dydz-\sum_{l=1}^{N}\int \int
\int_{Y\times \mathcal{T}\times Z}a_{il}\left( \frac{\partial \mathcal{\chi }%
_{jh}^{k}}{\partial y_{l}}+\frac{\partial \mathcal{\eta }_{jh}^{k}}{\partial
z_{l}}\right) dydzd\tau \text{,}  \label{eq3.11a}
\end{equation}%
where $\mathbf{\chi }_{jh}=\left( \mathcal{\chi }_{jh}^{k}\right) $ and $%
\mathbf{\eta }_{jh}=\left( \mathcal{\eta }_{jh}^{k}\right) $ are defined by (%
\ref{eq3.9b}). To the coefficients $q_{ijkh}$, we associate the differential
operator $\mathcal{Q}$ on $Q$ sending $\mathcal{D}^{\prime }\left( Q\right)
^{N}$ to $\mathcal{D}^{\prime }\left( Q\right) ^{N}$ in the following manner:%
\begin{equation}
\left\{ 
\begin{array}{c}
\text{for }\mathbf{z=}\left( z^{h}\right) \in \mathcal{D}^{\prime }\left(
Q\right) ^{N}\qquad \qquad \qquad \qquad \\ 
\left( \mathcal{Q}\mathbf{z}\right) ^{k}=-\sum_{i,j,h=1}^{N}q_{ijkh}\frac{%
\partial ^{2}z^{h}}{\partial x_{i}\partial x_{j}}\quad \left( 1\leq k\leq
N\right) \text{.}%
\end{array}%
\right.  \label{eq3.11b}
\end{equation}%
Next, we consider the Cauchy-Dirichlet boundary value problem%
\begin{equation}
\frac{\partial \mathbf{u}_{0}}{\partial t}+\mathcal{Q}\mathbf{u}%
_{0}+\sum_{j=1}^{N}u_{0}^{j}\frac{\partial \mathbf{u}_{0}}{\partial x_{j}}+%
\mathbf{grad}p_{0}=\mathbf{f}\text{ in }Q  \label{eq3.12a}
\end{equation}%
\begin{equation}
\func{div}\mathbf{u}_{0}=0\text{ in }Q  \label{eq3.12b}
\end{equation}%
\begin{equation}
\mathbf{u}_{0}=0\text{ on }\partial \Omega \times ]0,T[  \label{eq3.12c}
\end{equation}%
\begin{equation}
\mathbf{u}_{0}\left( 0\right) =0\text{\ in }\Omega \text{.}  \label{eq3.12d}
\end{equation}

\begin{proposition}
\label{pr3.2} Suppose $N=2$ and the hypotheses of Theorem \ref{th3.1} are
satisfied. Then, the boundary value problem (\ref{eq3.12a})-(\ref{eq3.12d})
admits at most one solution $\left( \mathbf{u}_{0},p_{0}\right) $ with $%
\mathbf{u}_{0}\in \mathcal{W}\left( 0,T\right) $ and $p_{0}\in L^{2}\left(
0,T;L^{2}\left( \Omega ;\mathbb{R}\right) \mathfrak{/}\mathbb{R}\right) $.
\end{proposition}

\begin{proof}
If $\left( \mathbf{u}_{0},p_{0}\right) \in \mathcal{W}\left( 0,T\right)
\times L^{2}\left( 0,T;L^{2}\left( \Omega ;\mathbb{R}\right) \right) $
verifies (\ref{eq3.12a})-(\ref{eq3.12d}) then 
\begin{equation*}
\begin{array}{c}
\int_{0}^{T}\left\langle \mathbf{u}_{0}^{\prime }\left( t\right) ,\mathbf{v}%
_{0}\left( t\right) \right\rangle dt+\sum_{i,j,k,h=1}^{N}\int_{Q}q_{ijkh}%
\frac{\partial u_{0}^{h}}{\partial x_{j}}\frac{\partial v_{0}^{k}}{\partial
x_{i}}dxdt+\int_{0}^{T}b\left( \mathbf{u}_{0}\left( t\right) ,\mathbf{u}%
_{0}\left( t\right) ,\mathbf{v}_{0}\left( t\right) \right) dt \\ 
=\int_{0}^{T}\left\langle \mathbf{f}\left( t\right) ,\mathbf{v}_{0}\left(
t\right) \right\rangle dt%
\end{array}%
\end{equation*}%
for all $\mathbf{v}_{0}\in L^{2}\left( 0,T;V\right) $. It follows from the
preceding equality that%
\begin{equation}
\begin{array}{c}
\int_{0}^{T}\left\langle \mathbf{u}_{0}^{\prime }\left( t\right) ,\mathbf{v}%
_{0}\left( t\right) \right\rangle dt+\sum_{i,j,k=1}^{N}\int \int \int
\int_{Q\times Y\times \mathcal{T}\times Z}a_{ij}\left( \frac{\partial
u_{0}^{k}}{\partial x_{j}}+\frac{\partial u_{1}^{k}}{\partial y_{j}}+\frac{%
\partial u_{2}^{k}}{\partial z_{j}}\right) \frac{\partial v_{0}^{k}}{%
\partial x_{i}}dxdtdydzd\tau \\ 
+\int_{0}^{T}b\left( \mathbf{u}_{0}\left( t\right) ,\mathbf{u}_{0}\left(
t\right) ,\mathbf{v}_{0}\left( t\right) \right) dt=\int_{0}^{T}\left\langle 
\mathbf{f}\left( t\right) ,\mathbf{v}_{0}\left( t\right) \right\rangle dt%
\end{array}
\label{eq3.13a}
\end{equation}%
for all $\mathbf{v}_{0}\in L^{2}\left( 0,T;V\right) $, where $%
u_{1}^{k}\left( x,t\right) =-\sum_{i,h=1}^{N}\frac{\partial u_{0}^{h}}{%
\partial x_{i}}\left( x,t\right) \mathcal{\chi }_{ih}^{k}$ and $%
u_{2}^{k}\left( x,t\right) =-\sum_{i,h=1}^{N}\frac{\partial u_{0}^{h}}{%
\partial x_{i}}\left( x,t\right) \mathcal{\eta }_{ih}^{k}$. Let us check
that $\mathbf{u=}\left( \mathbf{u}_{0},\mathbf{u}_{1},\mathbf{u}_{2}\right) $
(with $\mathbf{u}_{1}=\left( u_{1}^{k}\right) $ and $\mathbf{u}_{2}=\left(
u_{2}^{k}\right) $) verifies (\ref{eq3.1a})-(\ref{eq3.1c}). As $\left( 
\mathbf{u}_{1}\left( x,t\right) ,\mathbf{u}_{2}\left( x,t\right) \right) $
is solution to the variational equation (\ref{eq3.10c}) for fixed $\left(
x,t\right) $ in $Q$, we have 
\begin{equation}
\sum_{i,j,k=1}^{N}\int \int \int \int_{Q\times Y\times \mathcal{T}\times Z}%
\widehat{a}_{ij}\left( \frac{\partial u_{0}^{k}}{\partial x_{j}}+\frac{%
\partial u_{1}^{k}}{\partial y_{j}}+\frac{\partial u_{2}^{k}}{\partial z_{j}}%
\right) \left( \frac{\partial v_{1}^{k}}{\partial y_{i}}+\frac{\partial
v_{2}^{k}}{\partial z_{i}}\right) dxdtdydzd\tau =0  \label{eq3.13b}
\end{equation}%
for all $\left( \mathbf{v}_{1},\mathbf{v}_{2}\right) \in L^{2}\left(
Q;L_{per}^{2}\left( \mathcal{T};W_{y}\right) \right) \times L^{2}\left(
Q;L_{per}^{2}\left( \mathcal{T};L_{per}^{2}\left( Y;W_{z}\right) \right)
\right) $. Thus, by (\ref{eq3.13a})-(\ref{eq3.13b}) one has (\ref{eq3.1c})
with (\ref{eq3.1a}) and (\ref{eq3.1b}), of course. Therefore, we have the
unicity of $\mathbf{u=}\left( \mathbf{u}_{0},\mathbf{u}_{1},\mathbf{u}%
_{2}\right) $ in virtue of Lemma \ref{lem3.2}. It follows that $\left( 
\mathbf{u}_{0},p_{0}\right) $ is unique in $\mathcal{W}\left( 0,T\right)
\times L^{2}\left( 0,T;L^{2}\left( \Omega ;\mathbb{R}\right) \mathfrak{/}%
\mathbb{R}\right) $.
\end{proof}

\begin{theorem}
\label{th3.2} Suppose that the hypotheses of Theorem \ref{th3.1} are
satisfied. For $0<\varepsilon <1$, let $\left( \mathbf{u}_{\varepsilon
},p_{\varepsilon }\right) \in \mathcal{W}\left( 0,T\right) \times
L^{2}\left( 0,T;L^{2}\left( \Omega ;\mathbb{R}\right) \mathfrak{/}\mathbb{R}%
\right) $ be defined by (\ref{eq1.2a})-(\ref{eq1.2d}). Then, as $\varepsilon
\rightarrow 0$, $\mathbf{u}_{\varepsilon }\rightarrow \mathbf{u}_{0}$ in $%
\mathcal{W}\left( 0,T\right) $-weak and $p_{\varepsilon }\rightarrow p_{0}$
in $L^{2}\left( 0,T;L^{2}\left( \Omega \right) \right) $-weak, where $\left( 
\mathbf{u}_{0},p_{0}\right) \in \mathcal{W}\left( 0,T\right) \times
L^{2}\left( 0,T;L^{2}\left( \Omega ;\mathbb{R}\right) \mathfrak{/}\mathbb{R}%
\right) $ is the unique solution to (\ref{eq3.12a})-(\ref{eq3.12d}).
\end{theorem}

\begin{proof}
Let $E$ be a fundamental sequence. As in the proof of Theorem \ref{th3.1}, a
subsequence $E^{\prime }$ can be extracted from $E$ such that as $E^{\prime
}\ni \varepsilon \rightarrow 0$, one has (\ref{eq3.6a})-(\ref{eq3.6b}) and (%
\ref{eq3.6c}) with $\mathbf{u=}\left( \mathbf{u}_{0},\mathbf{u}_{1},\mathbf{u%
}_{2}\right) \in L^{2}\left( 0,T;\mathbb{E}_{0}^{\tau }\right) $. Therefore,
by (\ref{eq3.6c})) we have $p_{\varepsilon }\rightarrow p_{0}$ in $%
L^{2}\left( 0,T;L^{2}\left( \Omega \right) \right) $-weak as $E^{\prime }\ni
\varepsilon \rightarrow 0$, where $p_{0}$ is the mean value of $p$. It
follows that $p_{0}\in L^{2}\left( 0,T;L^{2}\left( \Omega ;\mathbb{R}\right) 
\mathfrak{/}\mathbb{R}\right) $. Furher, we have (\ref{eq3.8b}) for all%
\textbf{\ }$\mathbf{v}=\left( \mathbf{v}_{0},\mathbf{v}_{1},\mathbf{v}%
_{2}\right) \in \mathbb{E}_{0}^{1}$. Hence, taking a particular $\mathbf{v}%
=\left( \mathbf{v}_{0},\mathbf{v}_{1},\mathbf{v}_{2}\right) \in \mathbb{E}%
_{0}^{1}$ with $\mathbf{v}_{1}=0$ and $\mathbf{v}_{2}=0$ in (\ref{eq3.8b}),
and using (\ref{eq3.10a})-(\ref{eq3.10b}) we arrive at (\ref{eq3.12a}).
Moreover, $E$ is chosen arbitrary. Thus, by Proposition \ref{pr3.2}, the
proof is complete.
\end{proof}

Now, let us give a suitable form of the homogenized coefficients $q_{ijkh}$.
For this end, we introduce the vector space $\mathcal{M}_{y}$ of $\mathbf{F}%
=\left( F^{ij}\right) _{1\leq i,j\leq N}$ with $L_{per}^{2}\left( Y\times 
\mathcal{T};\mathbb{R}\right) $, and the vector space $\mathcal{M}_{z}$ of $%
\mathbf{G}=\left( G^{ij}\right) _{1\leq i,j\leq N}$ with $G^{ij}\in
L_{per}^{2}\left( Y\times \mathcal{T}\times Z;\mathbb{R}\right) $. We denote
by $\mathcal{M}_{z}\mathfrak{/}\mathbb{C}$ the vector subspace of $\mathcal{M%
}_{z}$ consisted of $\mathbf{G}=\left( G^{ij}\right) _{1\leq i,j\leq N}$
with $G^{ij}\in L_{per}^{2}\left( \mathcal{T};L_{per}^{2}\left(
Y;L_{per}^{2}\left( Z;\mathbb{R}\right) \mathfrak{/}\mathbb{C}\right)
\right) $. We provide $\mathcal{M}_{y}$ and $\mathcal{M}_{z}$ with the norms 
\begin{equation*}
\left\Vert \mathbf{F}\right\Vert _{\mathcal{M}_{y}}=\left(
\sum_{i,j=1}^{N}\left\Vert F^{ij}\right\Vert _{L_{per}^{2}\left( Y\times 
\mathcal{T}\right) }^{2}\right) ^{\frac{1}{2}}\text{\qquad }\left( \mathbf{F=%
}\left( F^{ij}\right) \in \mathcal{M}_{y}\right)
\end{equation*}%
and%
\begin{equation*}
\left\Vert \mathbf{G}\right\Vert _{\mathcal{M}_{z}}=\left(
\sum_{i,j=1}^{N}\left\Vert G^{ij}\right\Vert _{L_{per}^{2}\left( Y\times 
\mathcal{T}\times Z\right) }^{2}\right) ^{\frac{1}{2}}\text{\qquad }\left( 
\mathbf{G=}\left( G^{ij}\right) \in \mathcal{M}_{z}\right)
\end{equation*}%
respectively.

For $\left( \mathbf{F}_{1},\mathbf{F}_{2}\right) $ and $\left( \mathbf{G}%
_{1},\mathbf{G}_{2}\right) \in \mathcal{M}_{y}\times \mathcal{M}_{z}$ with $%
\mathbf{F}_{1}\mathbf{=}\left( F_{1}^{ij}\right) $, $\mathbf{F}_{2}=\left(
F_{2}^{ij}\right) $, $\mathbf{G}_{1}=\left( G_{1}^{ij}\right) $ and $\mathbf{%
G}_{2}=\left( G_{2}^{ij}\right) $, we put%
\begin{equation*}
A\left( \left( \mathbf{F}_{1},\mathbf{F}_{2}\right) ,\left( \mathbf{G}_{1},%
\mathbf{G}_{2}\right) \right) =\sum_{i,j,k=1}^{N}\int \int \int_{Y\times 
\mathcal{T}\times Z}a_{ij}\left( F_{1}^{jk}+F_{2}^{jk}\right) \left(
G_{1}^{ik}+G_{2}^{ik}\right) dydzd\tau \text{.}
\end{equation*}%
This defines a bilinear form $A\left( ,\right) $ on $\left( \mathcal{M}%
_{y}\times \mathcal{M}_{z}\right) \times \left( \mathcal{M}_{y}\times 
\mathcal{M}_{z}\right) $ which is continuous symmetric and $\mathcal{M}%
_{y}\times \left( \mathcal{M}_{z}\mathfrak{/}\mathbb{C}\right) $-coercive
with 
\begin{equation}
A\left( \left( \mathbf{F}_{1},\mathbf{F}_{2}\right) ,\left( \mathbf{F}_{1},%
\mathbf{F}_{2}\right) \right) \geq \alpha \left\Vert \left( \mathbf{F}_{1},%
\mathbf{F}_{2}\right) \right\Vert _{\mathcal{M}_{y}\times \mathcal{M}%
_{z}}^{2}\qquad \left( \left( \mathbf{F}_{1},\mathbf{F}_{2}\right) \in 
\mathcal{M}_{y}\times \left( \mathcal{M}_{z}\mathfrak{/}\mathbb{C}\right)
\right) \text{,}  \label{eq3.14a}
\end{equation}%
where $\alpha $ is the constant in (\ref{eq1.1c}) and 
\begin{equation*}
\left\Vert \left( \mathbf{F}_{1},\mathbf{F}_{2}\right) \right\Vert _{%
\mathcal{M}_{y}\times \mathcal{M}_{z}}=\left( \left\Vert \mathbf{F}%
_{1}\right\Vert _{\mathcal{M}_{y}}^{2}+\left\Vert \mathbf{F}_{2}\right\Vert
_{\mathcal{M}_{z}}^{2}\right) ^{\frac{1}{2}}\text{ for }\left( \mathbf{F}%
_{1},\mathbf{F}_{2}\right) \in \mathcal{M}_{y}\times \mathcal{M}_{z}\text{.}
\end{equation*}%
For $\mathbf{u}_{1}=\left( u_{1}^{k}\right) \in E_{y}$ and $\mathbf{u}%
_{2}=\left( u_{2}^{k}\right) \in E_{z}$ we put 
\begin{equation*}
\nabla _{y}\mathbf{u}_{1}=\left( \frac{\partial u_{1}^{k}}{\partial y_{j}}%
\right) _{1\leq j,k\leq N}\text{ et }\nabla _{z}\mathbf{u}_{2}=\left( \frac{%
\partial u_{2}^{k}}{\partial z_{j}}\right) _{1\leq j,k\leq N}\text{.}
\end{equation*}%
Then $(\nabla _{y}\mathbf{u}_{1},\nabla _{z}\mathbf{u}_{2})\in \mathcal{M}%
_{y}\times \mathcal{M}_{z}$, and 
\begin{equation}
a\left( \left( \mathbf{u}_{1},\mathbf{u}_{2}\right) ,\left( \mathbf{v}_{1},%
\mathbf{v}_{2}\right) \right) =A((\nabla _{y}\mathbf{u}_{1},\nabla _{z}%
\mathbf{u}_{2}),(\nabla _{y}\mathbf{v}_{1},\nabla _{z}\mathbf{v}_{2}))
\label{eq3.14b}
\end{equation}%
for all $\left( \mathbf{u}_{1},\mathbf{u}_{2}\right) $ and\ $\left( \mathbf{v%
}_{1},\mathbf{v}_{2}\right) \in E_{y}\times E_{z}$.

For $1\leq i,k\leq N$ we set 
\begin{equation*}
\mathbf{\theta }_{ik}=\left( \theta _{ik}^{lm}\right) _{1\leq l,m\leq N}%
\text{ with }\theta _{ik}^{lm}=\delta _{il}\delta _{km}\text{ for }%
l,m=1,...,N\text{,}
\end{equation*}%
where $\delta _{ij}$ is the Kronecker symbol. Hence, $\mathbf{\theta }%
_{ik}\in \mathcal{M}_{y}$\quad $\left( 1\leq i,k\leq N\right) $. As in \cite%
{bib2} (see also \cite{bib9}), one easily check (using (\ref{eq3.9b}) and (%
\ref{eq3.14b})) that%
\begin{equation}
q_{ijkh}=A((\nabla _{y}\mathbf{\chi }_{ik}-\mathbf{\theta }_{ik},\nabla _{z}%
\mathbf{\eta }_{ik}),(\nabla _{y}\mathbf{\chi }_{jh}-\mathbf{\theta }%
_{jh},\nabla _{z}\mathbf{\eta }_{jh}))\text{ and }q_{ijkh}=q_{jihk}\text{ }
\label{eq3.14c}
\end{equation}%
for $1\leq i,j,k,h\leq N$. Further, by (\ref{eq3.14a})-(\ref{eq3.14b}) we
show as in \cite{bib9} that the coefficients $q_{ijkh}$ verify 
\begin{equation*}
\sum_{i,j,k,h=1}^{N}q_{ijkh}\xi _{ik}\xi _{jh}\geq \alpha
_{0}\sum_{k,h=1}^{N}\left\vert \xi _{kh}\right\vert ^{2}
\end{equation*}%
for all $\mathbf{\xi =}\left( \xi _{ij}\right) $ with $\xi _{ij}\in \mathbb{R%
}$, where $\alpha _{0}>0$ is a constant.

\subsection{A corrector result for (\protect\ref{eq1.2a})-(\protect\ref%
{eq1.2d})}

Our goal in this subsection is to prove an approximation result for the
velocity $\mathbf{u}_{\varepsilon }$ when $\varepsilon \rightarrow 0$. We
assume that $\Omega $ is of class $\mathcal{C}^{2}$, and $\mathbf{f}$
belongs to $L^{\infty }\left( 0,T;H\right) $. Consequently, as the
homogenized coefficients $q_{ijkh}$ are constant (see (\ref{eq3.14c})), we
verify as in \cite[p.301, Theorem 3.6]{bib14} that the limit $\mathbf{u}_{0}$
solution to (\ref{eq3.12a})-(\ref{eq3.12d}) belongs to $L^{\infty }\left(
0,T;H^{2}\left( \Omega \right) ^{N}\right) $. Thus, the functions $\mathbf{u}%
_{1}$ and $\mathbf{u}_{2}$ given by Theorem \ref{th3.1} lie in $L^{2}\left(
0,T;H^{1}\left( \Omega \right) \right) \otimes L_{per}^{2}\left( \mathcal{T}%
;H_{\#}^{1}\left( Y\right) ^{N}\right) $ and $L^{2}\left( 0,T;H^{1}\left(
\Omega \right) \right) \otimes L_{per}^{2}\left( \mathcal{T}%
;L_{per}^{2}\left( Y;H_{\#}^{1}\left( Z\right) ^{N}\right) \right) $,
respectively, in virtue of (\ref{eq3.10a})-(\ref{eq3.10b}).

Before we state for our corrector result, let us prove the following
proposition.

\begin{proposition}
\label{pr3.3} Suppose that the hypotheses of Theorem \ref{th3.1} are
satisfied. Then as $\varepsilon \rightarrow 0$, 
\begin{equation*}
\nabla \mathbf{u}_{\varepsilon }\rightarrow \nabla \mathbf{u}_{0}+\nabla _{y}%
\mathbf{u}_{1}+\nabla _{z}\mathbf{u}_{2}\text{ reiteratively in }L^{2}\left(
Q\right) \text{-strong }\Sigma \text{, }
\end{equation*}%
i.e.,%
\begin{equation*}
\frac{\partial u_{\varepsilon }^{k}}{\partial x_{j}}\rightarrow \frac{%
\partial u_{0}^{k}}{\partial x_{j}}+\frac{\partial u_{1}^{k}}{\partial y_{j}}%
+\frac{\partial u_{2}^{k}}{\partial z_{j}}\text{ reiteratively in }%
L^{2}\left( Q\right) \text{-strong }\Sigma \text{ }\left( 1\leq j,k\leq
N\right) \text{.}
\end{equation*}
\end{proposition}

\begin{proof}
In view of (\ref{eq3.6a}) and Proposition \ref{pr2.5}, it remains to prove
that, as $\varepsilon \rightarrow 0$ 
\begin{equation}
\left\Vert \frac{\partial u_{\varepsilon }^{k}}{\partial x_{j}}\right\Vert
_{L^{2}\left( Q\right) }\rightarrow \left\Vert \frac{\partial u_{0}^{k}}{%
\partial x_{j}}+\frac{\partial u_{1}^{k}}{\partial y_{j}}+\frac{\partial
u_{2}^{k}}{\partial z_{j}}\right\Vert _{L^{2}\left( Q;L_{per}^{2}\left(
Y\times \mathcal{T}\times Z\right) \right) }\qquad \left( 1\leq j,k\leq
N\right) \text{.}  \label{eq3.15a}
\end{equation}%
For this purpose let $\eta >0$. Using the density of $E_{0}^{\infty }$ in $%
\mathbb{E}_{0}^{1}$ (see Lemma \ref{lem3.1}), we fix $\mathbf{\Phi =}\left( 
\mathbf{\psi }_{0},\mathbf{\psi }_{1},\mathbf{\psi }_{2}\right) $ as in (\ref%
{eq3.2a}) such that 
\begin{equation}
\int_{0}^{T}\widehat{a}_{\Omega }\left( \mathbf{u}\left( t\right) -\mathbf{%
\Phi }\left( t\right) ,\mathbf{u}\left( t\right) -\mathbf{\Phi }\left(
t\right) \right) dt\leq \frac{\alpha \eta ^{2}}{16}\text{,}  \label{eq3.15b}
\end{equation}%
where $\alpha $ is the constant in (\ref{eq1.1c}). Next, for any $%
\varepsilon >0$, let $\mathbf{\Phi }_{\varepsilon }$ be defined by (\ref%
{eq3.2b}).

(i) We begin by proving that there is some $\varepsilon _{1}>0$ such that 
\begin{equation}
\left\Vert \frac{\partial u_{\varepsilon }^{k}}{\partial x_{j}}-\frac{%
\partial \Phi _{\varepsilon }^{k}}{\partial x_{j}}\right\Vert _{L^{2}\left(
Q\right) }\leq \frac{\eta }{2}\qquad \left( 1\leq j,k\leq N\right)
\label{eq3.15c}
\end{equation}%
for all $0<\varepsilon \leq \varepsilon _{1}$. For this purpose we have 
\begin{equation*}
\int_{0}^{T}a^{\varepsilon }\left( \mathbf{u}_{\varepsilon }\left( t\right) -%
\mathbf{\Phi }_{\varepsilon }\left( t\right) ,\mathbf{u}_{\varepsilon
}\left( t\right) -\mathbf{\Phi }_{\varepsilon }\left( t\right) \right)
dt=\int_{0}^{T}\left( \mathbf{f}\left( t\right) ,\mathbf{u}_{\varepsilon
}\left( t\right) \right) dt-\int_{0}^{T}\left( \mathbf{u}_{\varepsilon
}^{\prime }\left( t\right) ,\mathbf{u}_{\varepsilon }\left( t\right) \right)
dt
\end{equation*}%
\begin{equation*}
-2\int_{0}^{T}a^{\varepsilon }\left( \mathbf{u}_{\varepsilon }\left(
t\right) ,\mathbf{\Phi }_{\varepsilon }\left( t\right) \right)
dt+\int_{0}^{T}a^{\varepsilon }\left( \mathbf{\Phi }_{\varepsilon }\left(
t\right) ,\mathbf{\Phi }_{\varepsilon }\left( t\right) \right) dt\text{.}
\end{equation*}%
By proceeding as in the proof of Theorem \ref{th3.1}, we deduce that as $%
\varepsilon \rightarrow 0$,%
\begin{equation*}
\int_{0}^{T}a^{\varepsilon }\left( \mathbf{u}_{\varepsilon }\left( t\right) -%
\mathbf{\Phi }_{\varepsilon }\left( t\right) ,\mathbf{u}_{\varepsilon
}\left( t\right) -\mathbf{\Phi }_{\varepsilon }\left( t\right) \right)
dt\rightarrow \int_{0}^{T}\left( \mathbf{f}\left( t\right) ,\mathbf{u}%
_{0}\left( t\right) \right) dt-\int_{0}^{T}\left( \mathbf{u}_{0}^{\prime
}\left( t\right) ,\mathbf{u}_{0}\left( t\right) \right) dt
\end{equation*}%
\begin{equation*}
-2\int_{0}^{T}\widehat{a}_{\Omega }\left( \mathbf{u}\left( t\right) ,\mathbf{%
\Phi }\left( t\right) \right) dt+\int_{0}^{T}\widehat{a}_{\Omega }\left( 
\mathbf{\Phi }\left( t\right) ,\mathbf{\Phi }\left( t\right) \right) dt
\end{equation*}%
\begin{equation*}
=\int_{0}^{T}\widehat{a}_{\Omega }\left( \mathbf{u}\left( t\right) ,\mathbf{u%
}\left( t\right) \right) dt-2\int_{0}^{T}\widehat{a}_{\Omega }\left( \mathbf{%
u}\left( t\right) ,\mathbf{\Phi }\left( t\right) \right) dt+\int_{0}^{T}%
\widehat{a}_{\Omega }\left( \mathbf{\Phi }\left( t\right) ,\mathbf{\Phi }%
\left( t\right) \right) dt\text{.}
\end{equation*}
Hence, there is some $\varepsilon _{1}>0$ such that 
\begin{equation*}
\int_{0}^{T}a^{\varepsilon }\left( \mathbf{u}_{\varepsilon }\left( t\right) -%
\mathbf{\Phi }_{\varepsilon }\left( t\right) ,\mathbf{u}_{\varepsilon
}\left( t\right) -\mathbf{\Phi }_{\varepsilon }\left( t\right) \right)
dt\leq \int_{0}^{T}\widehat{a}_{\Omega }\left( \mathbf{u}\left( t\right) -%
\mathbf{\Phi }\left( t\right) ,\mathbf{u}\left( t\right) -\mathbf{\Phi }%
\left( t\right) \right) dt+\frac{3\alpha \eta ^{2}}{16}
\end{equation*}%
for all $0<\varepsilon \leq \varepsilon _{1}$. But in view of (\ref{eq1.1c})
and (\ref{eq3.15b}), the preceding inequality leads to 
\begin{equation*}
\alpha \left\Vert \mathbf{u}_{\varepsilon }-\mathbf{\Phi }_{\varepsilon
}\right\Vert _{L^{2}\left( 0,T;H_{0}^{1}\left( \Omega \right) ^{N}\right)
}^{2}\leq \frac{\alpha \eta ^{2}}{16}+\frac{3\alpha \eta ^{2}}{16}=\frac{%
\alpha \eta ^{2}}{4}\text{,}
\end{equation*}%
for all $0<\varepsilon \leq \varepsilon _{1}$. Hence (\ref{eq3.15c}) follows
provided $0<\varepsilon \leq \varepsilon _{1}$.

(ii) Now, let $1\leq j,k\leq N$ be fixed freely. Thanks to (i), we have 
\begin{equation*}
\left\vert \left\Vert \frac{\partial u_{\varepsilon }^{k}}{\partial x_{j}}%
\right\Vert _{L^{2}\left( Q\right) }-\left\Vert \frac{\partial \Phi
_{\varepsilon }^{k}}{\partial x_{j}}\right\Vert _{L^{2}\left( Q\right)
}\right\vert \leq \frac{\eta }{2}\qquad \left( 0<\varepsilon \leq
\varepsilon _{1}\right) \text{.}
\end{equation*}%
On the other hand, by (\ref{eq3.15b}) with (\ref{eq1.1c}), we obtain 
\begin{equation*}
\alpha \int_{0}^{T}\left\Vert \mathbf{u}\left( t\right) -\mathbf{\Phi }%
\left( t\right) \right\Vert _{\mathbb{E}_{0}^{1,\tau }}^{2}dt\leq \frac{%
\alpha \eta ^{2}}{16}\text{,}
\end{equation*}%
i.e., 
\begin{equation*}
\left\Vert \nabla \mathbf{u}_{0}+\nabla _{y}\mathbf{u}_{1}+\nabla _{z}%
\mathbf{u}_{2}-\nabla \mathbf{\psi }_{0}-\nabla _{y}\mathbf{\psi }%
_{1}-\nabla _{z}\mathbf{\psi }_{2}\right\Vert _{L^{2}\left(
Q;L_{per}^{2}\left( Y\times \mathcal{T}\times Z\right) \right) ^{N^{2}}}\leq 
\frac{\eta }{4}\text{.}
\end{equation*}%
This implies 
\begin{equation*}
\left\vert \left\Vert \frac{\partial u_{0}^{k}}{\partial x_{j}}+\frac{%
\partial u_{1}^{k}}{\partial y_{j}}+\frac{\partial u_{2}^{k}}{\partial z_{j}}%
\right\Vert _{L^{2}\left( Q;L_{per}^{2}\left( Y\times \mathcal{T}\times
Z\right) \right) }-\left\Vert \frac{\partial \psi _{0}^{k}}{\partial x_{j}}+%
\frac{\partial \psi _{1}^{k}}{\partial y_{j}}+\frac{\partial \psi _{2}^{k}}{%
\partial z_{j}}\right\Vert _{L^{2}\left( Q;L_{per}^{2}\left( Y\times 
\mathcal{T}\times Z\right) \right) }\right\vert \leq \frac{\eta }{4}\text{.}
\end{equation*}%
Moreover, 
\begin{equation*}
\frac{\partial \Phi _{\varepsilon }^{k}}{\partial x_{j}}\rightarrow \frac{%
\partial \psi _{0}^{k}}{\partial x_{j}}+\frac{\partial \psi _{1}^{k}}{%
\partial y_{j}}+\frac{\partial \psi _{2}^{k}}{\partial z_{j}}\text{
reiteratively in }L^{2}\left( Q\right) \text{-strong }\Sigma
\end{equation*}%
as $\varepsilon \rightarrow 0$. Thus 
\begin{equation*}
\left\Vert \frac{\partial \Phi _{\varepsilon }^{k}}{\partial x_{j}}%
\right\Vert _{L^{2}\left( Q\right) }\rightarrow \left\Vert \frac{\partial
\psi _{0}^{k}}{\partial x_{j}}+\frac{\partial \psi _{1}^{k}}{\partial y_{j}}+%
\frac{\partial \psi _{2}^{k}}{\partial z_{j}}\right\Vert _{L^{2}\left(
Q;L_{per}^{2}\left( Y\times \mathcal{T}\times Z\right) \right) }
\end{equation*}%
as $\varepsilon \rightarrow 0$, and there is some $\varepsilon _{2}>0$ such
that 
\begin{equation*}
\left\vert \left\Vert \frac{\partial \Phi _{\varepsilon }^{k}}{\partial x_{j}%
}\right\Vert _{L^{2}\left( Q\right) }-\left\Vert \frac{\partial \psi _{0}^{k}%
}{\partial x_{j}}+\frac{\partial \psi _{1}^{k}}{\partial y_{j}}+\frac{%
\partial \psi _{2}^{k}}{\partial z_{j}}\right\Vert _{L^{2}\left(
Q;L_{per}^{2}\left( Y\times \mathcal{T}\times Z\right) \right) }\right\vert
\leq \frac{\eta }{4}
\end{equation*}%
for all $0<\varepsilon \leq \varepsilon _{2}$. Therefore, 
\begin{equation*}
\left\vert \left\Vert \frac{\partial u_{\varepsilon }^{k}}{\partial x_{j}}%
\right\Vert _{L^{2}\left( Q\right) }-\left\Vert \frac{\partial u_{0}^{k}}{%
\partial x_{j}}+\frac{\partial u_{1}^{k}}{\partial y_{j}}+\frac{\partial
u_{2}^{k}}{\partial z_{j}}\right\Vert _{L^{2}\left( Q;L_{per}^{2}\left(
Y\times \mathcal{T}\times Z\right) \right) }\right\vert \leq \left\vert
\left\Vert \frac{\partial u_{\varepsilon }^{k}}{\partial x_{j}}\right\Vert
_{L^{2}\left( Q\right) }-\left\Vert \frac{\partial \Phi _{\varepsilon }^{k}}{%
\partial x_{j}}\right\Vert _{L^{2}\left( Q\right) }\right\vert
\end{equation*}%
\begin{equation*}
+\left\vert \left\Vert \frac{\partial \Phi _{\varepsilon }^{k}}{\partial
x_{j}}\right\Vert _{L^{2}\left( Q\right) }-\left\Vert \frac{\partial \psi
_{0}^{k}}{\partial x_{j}}+\frac{\partial \psi _{1}^{k}}{\partial y_{j}}+%
\frac{\partial \psi _{2}^{k}}{\partial z_{j}}\right\Vert _{L^{2}\left(
Q;L_{per}^{2}\left( Y\times \mathcal{T}\times Z\right) \right) }\right\vert
\end{equation*}%
\begin{equation*}
+\left\vert \left\Vert \frac{\partial u_{0}^{k}}{\partial x_{j}}+\frac{%
\partial u_{1}^{k}}{\partial y_{j}}+\frac{\partial u_{2}^{k}}{\partial z_{j}}%
\right\Vert _{L^{2}\left( Q;L_{per}^{2}\left( Y\times \mathcal{T}\times
Z\right) \right) }-\left\Vert \frac{\partial \psi _{0}^{k}}{\partial x_{j}}+%
\frac{\partial \psi _{1}^{k}}{\partial y_{j}}+\frac{\partial \psi _{2}^{k}}{%
\partial z_{j}}\right\Vert _{L^{2}\left( Q;L_{per}^{2}\left( Y\times 
\mathcal{T}\times Z\right) \right) }\right\vert \leq \eta
\end{equation*}%
for all $0<\varepsilon \leq \varepsilon _{0}=\min \left( \varepsilon
_{1},\varepsilon _{2}\right) $. Hence, (\ref{eq3.15a}) is proved and the
proposition follows.
\end{proof}

Now, let us state for our corrector result.

\begin{theorem}
\label{th3.3} Suppose that the hypotheses of Theorem \ref{th3.1} are
satisfied. Suppose also that the functions $\mathbf{\chi }_{ik}$ and $%
\mathbf{\eta }_{ik}$ in (\ref{eq3.9b}) belong to $\mathcal{C}_{per}\left( 
\mathcal{T};\mathcal{C}_{per}^{1}\left( Y\right) \mathfrak{/}\mathbb{C}%
\right) ^{N}$ and $\mathcal{C}_{per}\left( \mathcal{T};\mathcal{C}%
_{per}^{1}\left( Y;\mathcal{C}_{per}^{1}\left( Z\right) \mathfrak{/}\mathbb{C%
}\right) \right) ^{N}$, respectively. Then, under the assumptions stated at
the beginning of this subsection, as $\varepsilon \rightarrow 0$ 
\begin{equation*}
\left\Vert \nabla \mathbf{u}_{\varepsilon }-\nabla \mathbf{u}%
_{0}-\varepsilon \nabla \mathbf{u}_{1}^{\varepsilon }-\varepsilon ^{2}\nabla 
\mathbf{u}_{2}^{\varepsilon }\right\Vert _{L^{2}\left( Q\right)
^{N^{2}}}\rightarrow 0\text{.}
\end{equation*}
\end{theorem}

\begin{proof}
As the functions $\left( \mathbf{\chi }_{ik},\mathbf{\eta }_{ik}\right) $
lie in $\mathcal{C}_{per}\left( \mathcal{T};\mathcal{C}_{per}^{1}\left(
Y\right) \mathfrak{/}\mathbb{C}\right) ^{N}\times \mathcal{C}_{per}\left( 
\mathcal{T};\mathcal{C}_{per}^{1}\left( Y;\mathcal{C}_{per}^{1}\left(
Z\right) \mathfrak{/}\mathbb{C}\right) \right) ^{N}$, 
\begin{equation*}
\mathbf{u}_{1}=-\sum_{i,k=1}^{N}\frac{\partial u_{0}^{k}}{\partial x_{i}}%
\mathbf{\chi }_{ik}\text{ \ and }\mathbf{u}_{2}=-\sum_{i,k=1}^{N}\frac{%
\partial u_{0}^{k}}{\partial x_{i}}\mathbf{\eta }_{ik}
\end{equation*}%
belong to $L^{2}\left( 0,T;H^{1}\left( \Omega \right) ^{N}\right) \otimes 
\mathcal{C}_{per}\left( \mathcal{T};\mathcal{C}_{per}^{1}\left( Y\right) 
\mathfrak{/}\mathbb{C}\right) ^{N}$ and $L^{2}\left( 0,T;H^{1}\left( \Omega
\right) ^{N}\right) \otimes \mathcal{C}_{per}\left( \mathcal{T};\mathcal{C}%
_{per}^{1}\left( Y;\mathcal{C}_{per}^{1}\left( Z\right) \mathfrak{/}\mathbb{C%
}\right) \right) ^{N}$, respectively. Thus for all $1\leq l\leq N$, $%
u_{1}^{l}=-\sum_{i,k=1}^{N}\frac{\partial u_{0}^{k}}{\partial x_{i}}\mathcal{%
\chi }_{ik}^{l}$ belongs to $L^{2}\left( 0,T;H^{1}\left( \Omega \right)
\right) \otimes \mathcal{C}_{per}\left( \mathcal{T};\mathcal{C}%
_{per}^{1}\left( Y\right) \mathfrak{/}\mathbb{C}\right) $ and $%
u_{2}^{l}=-\sum_{i,k=1}^{N}\frac{\partial u_{0}^{k}}{\partial x_{i}}\mathcal{%
\eta }_{ik}^{l}$ lies in $L^{2}\left( 0,T;H^{1}\left( \Omega \right) \right)
\otimes \mathcal{C}_{per}\left( \mathcal{T};\mathcal{C}_{per}^{1}\left( Y;%
\mathcal{C}_{per}^{1}\left( Z\right) \mathfrak{/}\mathbb{C}\right) \right) $%
. Therefore in view of Proposition \ref{pr3.3}, we apply Theorem \ref{th2.3}%
. This leads to 
\begin{equation*}
\left\Vert \frac{\partial }{\partial x_{j}}\left( u_{\varepsilon
}^{l}-u_{0}^{l}-\varepsilon \left( u_{1}^{l}\right) ^{\varepsilon
}-\varepsilon ^{2}\left( u_{2}^{l}\right) ^{\varepsilon }\right) \right\Vert
_{L^{2}\left( Q\right) }\rightarrow 0\qquad \left( 1\leq l,j\leq N\right)
\end{equation*}%
as $\varepsilon \rightarrow 0$. The theorem is proved.
\end{proof}

\noindent \textbf{Conclusion}. As a concluding remark, the time dependent
viscosities of the Navier-Stokes type flows have not been taken in account
in this study, in view of the lack of unicity of solutions to (\ref{eq1.2a}%
)-(\ref{eq1.2d}) and their estimates. It is not an easy task to prove the
unicity and estimates in \cite[Proposition 2.1]{bib11}, when the
coefficients $a_{ij}$ depend on the time variable. However, this
homogenization process remains valid provided that the estimates in \cite[%
Proposition 2.1]{bib11} and the unicity are established.

\end{document}